\documentclass{amsart}
\usepackage{latexsym,amsfonts,amssymb,exscale,enumerate,comment}
\usepackage{amsthm,amscd}
\usepackage{amsmath}
\usepackage{hyperref}
\usepackage{accents}
\usepackage{url}

\input xy
\usepackage[all]{xy}
\xyoption{line}
\xyoption{arrow}
\xyoption{color}
\SelectTips{cm}{}
\usepackage{tikz}
\usetikzlibrary{shapes,snakes}
\usetikzlibrary{decorations.markings}
\usetikzlibrary{decorations.pathreplacing}
\tikzstyle directed=[postaction={decorate,decoration={markings,mark=at position #1 with {\arrow{>}}}}]

\newcommand{\hackcenter}[1]{\xy (0,0)*{#1}; \endxy}

\tikzset{->-/.style={decoration={markings, mark=at position #1 with {\arrow{>}}},postaction={decorate}}}

\tikzset{middlearrow/.style={ decoration={markings,mark= at position 0.5 with {\arrow{#1}} , },
postaction={decorate}}}

\usepackage{graphicx}
\usepackage{color}

\newcommand{\scs}{\scriptstyle}

\theoremstyle{plain}
\newtheorem{theorem}{Theorem}

\newtheorem{proposition}[theorem]{Proposition}

\theoremstyle{definition}

\newtheorem{definition}[theorem]{Definition}

\newtheorem{remark}[theorem]{Remark}

\numberwithin{equation}{section}
\numberwithin{theorem}{section}

\numberwithin{equation}{section}
\let\tilde=\widetilde

\usepackage{bbm}

\def\Z{{\mathbbm Z}}

\title{Khovanov-Lauda-Rouquier subalgebras and redotted Webster algebras}
\author{Yasuyoshi Yonezawa}
\email{yasuyoshi.yonezawa@gmail.com}
\address{Advanced Mathematical Institute \\
Osaka City University\\
Osaka, Japan}

\begin{document}
\maketitle

\begin{abstract}
We define Khovanov-Lauda-Rouquier subalgebras which are generalizations of redotted versions of Webster's tensor product algebras of type $A_1$ defined in \cite{KLSY}.
Quotient algebras of these subalgebras are isomorphic to Webster's tensor product algebras in general type.
\end{abstract}

\section{Introduction}

\subsection{Khovanov-Lauda-Rouquier algebras, the redotted Webster algebras and categorical braid group action}
Khovanov-Lauda and Rouquier defined algebras categorifying quantum groups associated to semisimple Lie algebras $\mathfrak{g}$ and their irreducible representations \cite{KL1, KL2, KL3, Rou2}.
In the case of  $\mathfrak{g} = \mathfrak{sl}_2$, the algebra is isomorphic to the nilHecke algebra.
In Khovanov-Lauda diagrammatic approach, generators of the algebra correspond to planar diagrams consisting of strands labelled by the simple roots of the Lie algebra $\mathfrak{g}$.

Webster defined algebras categorifying the tensor product $V_{\lambda_1}\otimes V_{\lambda_2}\otimes \cdots\otimes V_{\lambda_m}$ of irreducible representations of the quantum group associated to a semisimple Lie algebra $\mathfrak{g}$ \cite{Web}.
These algebras follow the Khovanov-Lauda diagrammatic approach of categorifications of quantum groups and their irreducible representations from \cite{Lau1, KL1, KL2, KL3}. 

In the context of $\mathfrak{sl}_2$, Khovanov-Sussan studied a deformation type of these tensor product algebras. In their context, the deformation led to additional dot-generators on red strands and additional diagrammatic relations \cite{KS}.
The algebras were called the redotted Webster algebra $W(n, k)$ in \cite{KS}, where there are $k$ black strands and $n$ red strands labelled by the fundamental representation of $\mathfrak{sl}_2$. 
These algebras were studied for constructing Webster’s categorical braid invariants in the homotopy category.
The algebra $W(n, 0)$ and the homotopy category of their modules are related to Soergel bimodules and a categorical braid group action in the homotopy category defined by Rouquier \cite{Rou-Soergel}.
Khovanov, Lauda, Sussan and the author extended the algebra $W(n,k)$ to $W(\mathbf{s}, k)$ where $\mathbf{s} = (s_1, . . . , s_m)$ is a tuple of natural numbers corresponding to arbitrary symmetric powers of the fundamental representation of $\mathfrak{sl}_2$ \cite{KLSY}.
We also constructed the braid group action on the homotopy category of $W(\mathbf{s}, k)$. 
Subsequently, in the type $A_n$, Webster studied a generalization for $W(\mathbf{s}, k)$ associated with Gelfand-Tsetlin modules \cite{Web-An}.

In this paper, we study these algebras in the context of general type.
We extend the algebra $W(\mathbf{s}, k)$ to the algebra $S(\lambda,\nu)$, where $\lambda$ is a sequence of weights of irreducible representations and $\nu$ is a sequence of the simple roots of the Lie algebra $\mathfrak{g}$, as Khovanov-Lauda-Rouquier subalgebras.
It seems that our Khovanov-Lauda-Rouquier subalgebra $S_{C}(\lambda,\nu)$ with specialized scalar parameters $t(\tilde{i},\tilde{j})$ is related to Webster's generalization of the type $A_n$.

\subsection{Motivation for link homology theory}
The categorification of the Jones polynomial is defined in the homotopy category of complexes over an additive category in \cite{Kh1}. 

Since the Jones polynomial is a link invariant associated with the quantum group of $\mathfrak{sl}_2$ and its 2-dimensional irreducible representation, similar efforts constructing link homologies associated with quantum groups and their representations have been made in various approaches.

For the $\mathfrak{sl}_2$ link invariant, its link homology is studied using a functor value defined in \cite{Kh2}, using cobordisms\cite{Dror}, using geometric approach\cite{CaKa-sl2}, and using Howe duality approach\cite{LQR}.
For the coloured $\mathfrak{sl}_2$ link invariant, its link homology is studied using cabling\cite{Kh05}, using a categorification of the Jones-Wenzl projectors\cite{CoopK0,Roz}, using Lie theoretic approach\cite{SS}, using a polynomial action\cite{Hog} and using a tensor product algebra approach \cite{KLSY}.
For the $\mathfrak{sl}_n$ link invariant (associated with the $n$-dimensional irreducible representation), its link homology is studied using matrix factorizations\cite{KR}, and using geometric approach\cite{CK-slm}, using Lie theoretic approach\cite{MS-O}, and using cobordisms\cite{Web-cob, SY}.
For HOMFLY-PT polynomial, its link homology is studied using matrix factorizations\cite{KhR2}, using Soergel bimodules\cite{Kho-triple}, and using geometric approach\cite{WW}.
For coloured $\mathfrak{sl}_n$ link invariant (associated with anti-symmetric tensor product of the $n$-dimensional irreducible representation), its link homology is studied using matrix factorizations\cite{Yo,Wu}, using Lie theoretic approach\cite{Su}, and using Howe duality approach\cite{CKL-skew, MY}.
For the general $\mathfrak{g}$ link invariant, its link homology is studied using a tensor product algebra approach in the derived category\cite{Web}.

Using Khovanov-Lauda-Rouquier subalgebras introduced in this paper, we will construct $\mathfrak{g}$ link homological invariant in the homotopy category.

\bigskip
\noindent
{\bf Acknowledgements.}
The author would like to thank Joshua Sussan for carefully reading this manuscript and giving helpful comments.
The author would like to thank Mikhail Khovanov and Aaron Lauda for helpful comments.

\section{Khovanov-Lauda-Rouquier algebras}\label{klr}
We recall Khovanov-Lauda-Rouquier (quiver Hecke) algebras based on \cite{CLau} which is a categorification of one-half of the quantum group associated to an arbitrary Cartan datum, introduced by Khovanov-Lauda, Rouquier\cite{KL1,KL2,Rou2}.

\subsection{Cartan datum and scalar parameter}\label{cartan}
A Cartan datum $(I, \cdot)$ consists of a finite set $I=\{1,2,...,N\}$ and a  symmetric bilinear form on $\Z [I]$ subject to the conditions
\begin{itemize}
\item $i\cdot i\in 2\Z_{>0}$ for any $i\in I$,
\item $2\frac{i\cdot j}{i\cdot i} $ is a non positive integer for any $i\not= j$ in $I$.
\end{itemize}
The associated matrix	
$$
C=(c(i,j))_{i,j\in I}:=\left(2\frac{i\cdot j}{i\cdot i}\right)_{i,j\in I}
$$
is a symmetrizable generalized Cartan matrix.

We set scalar parameters
\begin{eqnarray*}
Q_C&=&
\{t(i,j)\in \Bbbk^{\times}\,|\,i,j\in I\}
\cup
\{r(i)\in \Bbbk^{\times}\,|\, i\in I \}
\\
&&\qquad \cup
\{s^{pq}(i,j)\in\Bbbk\,|\,i,j\in I,i\not=j, p,q\in \Z_{> 0}, p (i\cdot i)+q (j\cdot j)=-2 (i\cdot j)\}
\end{eqnarray*}
satisfying
\begin{eqnarray}
\nonumber
t(i,i)=1 \text{ for all }i\in I,\quad
t(i,j) = t(j,i) \text{ when } i\cdot j = 0,\quad
s^{pq}(i,j)=s^{qp}(j,i).
\end{eqnarray}

For convenience, we set 
$$
s^{0q}(i,j)=s^{p0}(i,j)=0
$$ 
and
$$
s^{pq}(i,j)=0 \text{ when } p (i\cdot i)+q (j\cdot j)\not=-2 (i\cdot j).
$$ 

\subsection{Khovanov-Lauda-Rouquier algebra $R_n(Q_C)$}
Here, we recall Khovanov-Lauda-Rouquier algebra $R_n(Q_C)$ associated with a Cartan datum $(I,\cdot)$ and the scalar parameters $Q_C$\cite{CLau}.

Let $n\geq 0$ be an integer.
For a sequence $\mathbf{i}=(i_1,...,i_{n})\in I^n$, denote by $\mathbf{i}_j$ the $j$-th entry of $\mathbf{i}$.
The symmetric group $\mathfrak{S}_n$ naturally acts on $I^n$. For an element $\sigma\in\mathfrak{S}_n$, $\sigma\cdot (i_1,...,i_{n})=(i_{\sigma(1)},...,i_{\sigma(n)})$.

The Khovanov-Lauda-Rouquier algebra $R_n(Q_C)$ is the $\Z$-graded $\Bbbk$-algebra generated by
\begin{equation}
e(\mathbf{i})\quad(\mathbf{i} \in I^n),\qquad
x_j\quad (1\leq j\leq n),\qquad
\psi_k\quad (1\leq k \leq n-1),
\end{equation}
satisfying the following relations.

\begin{minipage}{0.4\textwidth}
\begin{align}
& e(\mathbf{i})e(\mathbf{j}) = \delta_{\mathbf{i},\mathbf{j}}e(\mathbf{i})
\\
& x_j e(\mathbf{i})=e(\mathbf{i})x_j
\\
& \psi_j e(\mathbf{i}) = e(\sigma_j\cdot\mathbf{i})\psi_j
\end{align}
\end{minipage}
\begin{minipage}{0.6\textwidth}
\begin{align}
& x_j x_k=x_k x_j
\\
&\psi_j x_k = x_k\psi_j  \quad \text{if $k\not=j$ and $k\not=j+1$}
\\
& \psi_j\psi_k = \psi_k\psi_j \quad \text{if $|j - k| > 1$}
\end{align}
\end{minipage}
\begin{minipage}{\textwidth}
\begin{align}
\label{dot-slide1}
&x_j \psi_j e(\mathbf{i}) - \psi_j x_{j+1} e(\mathbf{i})= \left\{
\begin{array}{ll}r(i) e(\mathbf{i}) & \text{if $\mathbf{i}_j=\mathbf{i}_{j+1}$}\\
0 & \text{if $\mathbf{i}_j\not=\mathbf{i}_{j+1}$}
\end{array}
\right.
\\
\label{dot-slide2}
&\psi_j x_j e(\mathbf{i}) - x_{j+1} \psi_j e(\mathbf{i})= \left\{
\begin{array}{ll}r(i) e(\mathbf{i}) & \text{if $\mathbf{i}_j=\mathbf{i}_{j+1}$}\\
0 & \text{if $\mathbf{i}_j\not=\mathbf{i}_{j+1}$}
\end{array}
\right.
\end{align}
\end{minipage}
\begin{minipage}{\textwidth}
\begin{align}
\label{RII}
&\psi_j^2 e(\mathbf{i})=\left\{
	\begin{array}{ll}
		0
		&\text{if } \mathbf{i}_j=\mathbf{i}_{j+1}\\
		t(\mathbf{i}_{j},\mathbf{i}_{j+1})e(\mathbf{i})
		&\text{if } \mathbf{i}_j\cdot\mathbf{i}_{j+1}=0\\
		\displaystyle(t(\mathbf{i}_{j},\mathbf{i}_{j+1})x_{j}^{-c(\mathbf{i}_{j},\mathbf{i}_{j+1})}+t(\mathbf{i}_{j+1},\mathbf{i}_{j})x_{j+1}^{-c(\mathbf{i}_{j+1},\mathbf{i}_{j})}+\sum_{p,q}s^{pq}(\mathbf{i}_{j},\mathbf{i}_{j+1})x_j^p x_{j+1}^q)  e(\mathbf{i})
		&\text{if } \mathbf{i}_j\cdot\mathbf{i}_{j+1}<0
	\end{array}\right.
\\
\label{RIII}
& (\psi_j\psi_{j+1} \psi_j-\psi_{j+1}\psi_j \psi_{j+1}) e(\mathbf{i}) \\
\nonumber
& \qquad
=\left\{
	\begin{array}{ll}
	 \displaystyle 
	 r(\mathbf{i}_{j}) t(\mathbf{i}_{j},\mathbf{i}_{j+1})\sum_{d_1+d_2=-c(\mathbf{i}_{j},\mathbf{i}_{j+1})-1}
	 x_j^{d_1}x_{j+2}^{d_2}e(\mathbf{i})&\\
	 \displaystyle 
	 \hspace{2cm}+r(\mathbf{i}_{j}) \sum_{p,q}s^{pq}(\mathbf{i}_{j},\mathbf{i}_{j+1})\sum_{0\leq k_1,k_2\atop k_1+k_2=p-1}x_{j}^{k_1}x_{j+1}^q x_{j+2}^{k_2}
	&\text{if } \mathbf{i}_{j}=\mathbf{i}_{j+2} \text{ and } c(\mathbf{i}_{j},\mathbf{i}_{j+1})+1\leq 0 \\
	0
	&\text{otherwise}.
	\end{array}\right.
\end{align}
\end{minipage}

A $\Z$-grading of these generators is defined by
\[
\deg (e(\mathbf{i}))=0,\quad 
\deg (x_j e(\mathbf{i}))=\mathbf{i}_j\cdot \mathbf{i}_j,\quad 
\deg (\psi_j e(\mathbf{i}))=-\mathbf{i}_j\cdot \mathbf{i}_{j+1}.
\]

\subsection{Diagrammatic description}\label{diagram}
We have a diagrammatic description of Khovanov-Lauda-Rouquier algebra $R_n(Q_C)$.

The generator $e(\mathbf{i})$ for $\mathbf{i}=(i_1,i_2,...,i_n) \in I^n$ is represented by the diagram consisting of vertical strands whose $j$-th strand (counting from the left for each $j$) is labeled with $i_j$ as follows.
\[
e(\mathbf{i})=\hackcenter{\begin{tikzpicture}[scale=0.6]
    \node at (0,1.75) {};
    \draw[thick] (-1.5,0) -- (-1.5,1.5)  node[pos=.35, shape=coordinate](DOT){};
    \node at (-1.5,-.25) {$\scs  i_1$};
    \draw[thick] (0,0) -- (0,1.5)  node[pos=.35, shape=coordinate](DOT){};
    \node at (0,-.25) {$\scs  i_j$};
    \draw[thick] (1.5,0) -- (1.5,1.5)  node[pos=.35, shape=coordinate](DOT){};
    \node at (1.5,-.25) {$\scs  i_n$};
    \node at (-.75,.75) {$\cdots$};
    \node at (.75,.75) {$\cdots$};
\end{tikzpicture}}
\]
Dots on strands correspond to the generators $x_j$.
The generator $x_j e(\mathbf{i})$ is represented by the dot on the $j$-th strand counting from the left of the $n$-strand diagram as follows.
\begin{eqnarray*}
x_j e(\mathbf{i})&=&\hackcenter{\begin{tikzpicture}[scale=0.6]
    \node at (0,1.75) {};
    \draw[thick] (0,0) -- (0,1.5)  node[pos=.35, shape=coordinate](DOT){};
    \filldraw  (DOT) circle (2.5pt);
    \node at (0,-.25) {$\scs  i_j$};
    \draw[thick] (-1.5,0) -- (-1.5,1.5) ;
    \node at (-1.5,-.25) {$\scs  i_1$};
    \draw[thick] (1.5,0) -- (1.5,1.5) ;
    \node at (1.5,-.25) {$\scs  i_n$};
    \node at (-.75,.75) {$\cdots$};
    \node at (.75,.75) {$\cdots$};
\end{tikzpicture}}
\end{eqnarray*}

For simplicity, $x_j^d e(\mathbf{i})$ is represented by the dot labeled with $d$.
\begin{eqnarray*}
x_j^d e(\mathbf{i})&=&\hackcenter{\begin{tikzpicture}[scale=0.6]
    \node at (0,1.75) {};
    \draw[thick] (0,0) -- (0,1.5)  node[pos=.35, shape=coordinate](DOT){};
    \filldraw  (DOT) circle (2.5pt);
    \node at (.25,.5) {$\scs  d$};
    \node at (0,-.25) {$\scs  i_j$};
    \draw[thick] (-1.5,0) -- (-1.5,1.5) ;
    \node at (-1.5,-.25) {$\scs  i_1$};
    \draw[thick] (1.5,0) -- (1.5,1.5) ;
    \node at (1.5,-.25) {$\scs  i_n$};
    \node at (-.75,.75) {$\cdots$};
    \node at (.75,.75) {$\cdots$};
\end{tikzpicture}}
\end{eqnarray*}

A crossing of strands corresponds to the generator $\psi_j$.
The generator $\psi_j e(\mathbf{i})$ is represented by a crossing of the $j$-th and ($j+1$)-th strands as follows.
\[
\psi_j e(\mathbf{i})=
\hackcenter{\begin{tikzpicture}[scale=0.6]
    \draw[thick] (0,0) .. controls (0,.75) and (1.5,.75) .. (1.5,1.5);
    \draw[thick] (1.5,0) .. controls (1.5,.75) and (0,.75) .. (0,1.5);
    \draw[thick] (-1.5,0) -- (-1.5,1.5) ;
    \draw[thick] (3,0) -- (3,1.5);
    \node at (-1.5,-.25) {$\scs  i_1$};
    \node at (0,-.25) {$\scs  i_j$};
    \node at (1.5,-.25) {$\scs  i_{j+1}$};
    \node at (3,-.25) {$\scs  i_n$};
    \node at (-.75,.75) {$\cdots$};
    \node at (2.25,.75) {$\cdots$};
\end{tikzpicture}}
\]

The product in $R_n(Q_C)$ is given by concatenation of diagrams.
For elements $x,y \in R_n(Q_C)$ with diagrammatic descriptions $D_x$ and  $D_y$, the product $x y$ is represented by the diagram composed of $D_x$ and  $D_y$ whose the bottom of $D_x$ and the top of $D_y$ are connected.

For simplicity, when it causes no confusion, we will omit vertical strands in diagrams.
For instance, we depict $x_j e(\mathbf{i})$ and $\psi_j e(\mathbf{i})$ by
\[
\hackcenter{\begin{tikzpicture}[scale=0.6]
    \draw[thick] (0,0) -- (0,1.5)  node[pos=.35, shape=coordinate](DOT){};
    \filldraw  (DOT) circle (2.5pt);
    \node at (0,-.2) {$\scs i_j $};
\end{tikzpicture}}
\qquad,\qquad
\hackcenter{\begin{tikzpicture}[scale=0.6]
    \draw[thick] (0,0) .. controls (0,.75) and (1.5,.75) .. (1.5,1.5);
    \draw[thick] (1.5,0) .. controls (1.5,.75) and (0,.75) .. (0,1.5);
    \node at (0,-.25) {$\scs  i_j$};
    \node at (1.5,-.25) {$\scs  i_{j+1}$};
\end{tikzpicture}} .
\]

The degrees of the generating diagrams are
\[
\deg
\left( \;
\hackcenter{\begin{tikzpicture}[scale=0.6]
    \draw[thick] (0,0) -- (0,1.5)  node[pos=.35, shape=coordinate](DOT){};
    \filldraw  (DOT) circle (2.5pt);
    \node at (0,-.2) {$\scs i_j$};
\end{tikzpicture}}\;
\right)
= i_j\cdot i_j
\qquad
\deg\left( \; \hackcenter{\begin{tikzpicture}[scale=0.6]
    \draw[thick] (0,0) .. controls (0,.75) and (1.5,.75) .. (1.5,1.5);
    \draw[thick] (1.5,0) .. controls (1.5,.75) and (0,.75) .. (0,1.5);
    \node at (0,-.25) {$\scs  i_j$};
    \node at (1.5,-.25) {$\scs  i_{j+1}$};
\end{tikzpicture}} \; \right) = -i_j\cdot i_{j+1}.
\]
 
The algebra relations \eqref{dot-slide1} and \eqref{dot-slide2} are
\[
\hackcenter{\begin{tikzpicture}[scale=0.8]
    \draw[thick] (0,0) .. controls (0,.75) and (.75,.75) .. (.75,1.5)
        node[pos=.25, shape=coordinate](DOT){};
    \draw[thick](.75,0) .. controls (.75,.75) and (0,.75) .. (0,1.5);
    \filldraw  (DOT) circle (2.5pt);
    \node at (.0,-.2) {$\scs i$};
    \node at (.75,-.2) {$\scs i$};
\end{tikzpicture}}
\quad -\quad
\hackcenter{\begin{tikzpicture}[scale=0.8]
    \draw[thick] (0,0) .. controls (0,.75) and (.75,.75) .. (.75,1.5)
        node[pos=.75, shape=coordinate](DOT){};
    \draw[thick](.75,0) .. controls (.75,.75) and (0,.75) .. (0,1.5);
    \filldraw  (DOT) circle (2.5pt);
    \node at (.0,-.2) {$\scs i$};
    \node at (.75,-.2) {$\scs i$};
\end{tikzpicture}}
\quad=\quad
r(i)\hackcenter{\begin{tikzpicture}[scale=0.8]
    \draw[thick] (0,0) --(0,1.5);
    \draw[thick](.75,0) -- (.75,1.5);
    \node at (.0,-.2) {$\scs i$};
    \node at (.75,-.2) {$\scs i$};
\end{tikzpicture}}
\quad=\quad
\hackcenter{\begin{tikzpicture}[scale=0.8]
    \draw[thick](0,0) .. controls (0,.75) and (.75,.75) .. (.75,1.5);
    \draw[thick] (.75,0) .. controls (.75,.75) and (0,.75) .. (0,1.5)
        node[pos=.75, shape=coordinate](DOT){};
    \filldraw  (DOT) circle (2.75pt);
    \node at (.0,-.2) {$\scs i$};
    \node at (.75,-.2) {$\scs i$};
\end{tikzpicture}}
\quad-\quad
\hackcenter{\begin{tikzpicture}[scale=0.8]
    \draw[thick](0,0) .. controls (0,.75) and (.75,.75) .. (.75,1.5);
    \draw[thick] (.75,0) .. controls (.75,.75) and (0,.75) .. (0,1.5)
        node[pos=.25, shape=coordinate](DOT){};
      \filldraw  (DOT) circle (2.75pt);
    \node at (.0,-.2) {$\scs i$};
    \node at (.75,-.2) {$\scs i$};
\end{tikzpicture}}
\]

\[
\hackcenter{\begin{tikzpicture}[scale=0.8]
    \draw[thick] (0,0) .. controls (0,.75) and (.75,.75) .. (.75,1.5)
        node[pos=.25, shape=coordinate](DOT){};
    \draw[thick](.75,0) .. controls (.75,.75) and (0,.75) .. (0,1.5);
    \filldraw  (DOT) circle (2.5pt);
    \node at (.0,-.2) {$\scs i$};
    \node at (.75,-.2) {$\scs j$};
\end{tikzpicture}}
\quad =\quad
\hackcenter{\begin{tikzpicture}[scale=0.8]
    \draw[thick] (0,0) .. controls (0,.75) and (.75,.75) .. (.75,1.5)
        node[pos=.75, shape=coordinate](DOT){};
    \draw[thick](.75,0) .. controls (.75,.75) and (0,.75) .. (0,1.5);
    \filldraw  (DOT) circle (2.5pt);
    \node at (.0,-.2) {$\scs i$};
    \node at (.75,-.2) {$\scs j$};
\end{tikzpicture}}
\qquad,\qquad
\hackcenter{\begin{tikzpicture}[scale=0.8]
    \draw[thick](0,0) .. controls (0,.75) and (.75,.75) .. (.75,1.5);
    \draw[thick] (.75,0) .. controls (.75,.75) and (0,.75) .. (0,1.5)
        node[pos=.75, shape=coordinate](DOT){};
    \filldraw  (DOT) circle (2.75pt);
    \node at (.0,-.2) {$\scs i$};
    \node at (.75,-.2) {$\scs j$};
\end{tikzpicture}}
\quad=\quad
\hackcenter{\begin{tikzpicture}[scale=0.8]
    \draw[thick](0,0) .. controls (0,.75) and (.75,.75) .. (.75,1.5);
    \draw[thick] (.75,0) .. controls (.75,.75) and (0,.75) .. (0,1.5)
        node[pos=.25, shape=coordinate](DOT){};
      \filldraw  (DOT) circle (2.75pt);
    \node at (.0,-.2) {$\scs i$};
    \node at (.75,-.2) {$\scs j$};
\end{tikzpicture}}
\quad \text{for } i\not=j.
\]

The algebra relations \eqref{RII} are
\[
\hackcenter{\begin{tikzpicture}[scale=0.8]
    \draw[thick] (0,0) .. controls ++(0,.5) and ++(0,-.5) .. (.75,.75);
    \draw[thick] (.75,0) .. controls ++(0,.5) and ++(0,-.5) .. (0,.75);
    \draw[thick] (0,.75) .. controls ++(0,.5) and ++(0,-.5) .. (.75,1.5);
    \draw[thick] (.75,.75) .. controls ++(0,.5) and ++(0,-.5) .. (0,1.5);
    \node at (0,-.2) {$\scs i$};
    \node at (.8,-.2) {$\scs j$};
\end{tikzpicture}}
 \;\; = \;\;
\left\{
\begin{array}{ll}
0&\text{if } i=j\\[1em]
t(i,j)\hackcenter{\begin{tikzpicture}[scale=0.8]
    \draw[thick] (0,0) -- (0,1.5);
    \draw[thick] (.75,0) -- (.75,1.5);
    \node at (0,-.2) {$\scs i$};
    \node at (.8,-.2) {$\scs j$};
    \end{tikzpicture}}
&\text{if } i\cdot j=0\\
t(i,j)
\hackcenter{\begin{tikzpicture}[scale=0.8]
    \draw[thick] (0,0) -- (0,1.5)  node[pos=.55, shape=coordinate](DOT){};
    \filldraw  (DOT) circle (2.5pt);
    \draw[thick] (.75,0) -- (.75,1.5);
    \node at (0,-.2) {$\scs i$};
    \node at (.8,-.2) {$\scs j$};
    \node at (-.7,.75) {$\scs -c(i,j)$};
\end{tikzpicture}}
+
t(j,i)
\hackcenter{\begin{tikzpicture}[scale=0.8]
    \draw[thick] (0,0) -- (0,1.5);
    \draw[thick] (.75,0) -- (.75,1.5) node[pos=.55, shape=coordinate](DOT2){};
    \filldraw  (DOT2) circle (2.75pt);
    \node at (0,-.2) {$\scs i$};
    \node at (.8,-.2) {$\scs j$};
    \node at (1.5,.75) {$\scs -c(j,i)$};
\end{tikzpicture}}
+
\displaystyle\sum_{p,q}s^{pq}(i,j)
\hackcenter{\begin{tikzpicture}[scale=0.8]
    \draw[thick] (0,0) -- (0,1.5);
    \draw[thick] (.75,0) -- (.75,1.5) node[pos=.55, shape=coordinate](DOT2){};
    \filldraw  (DOT) circle (2.5pt);
    \filldraw  (DOT2) circle (2.75pt);
    \node at (0,-.2) {$\scs i$};
    \node at (.8,-.2) {$\scs j$};
    \node at (1.1,.75) {$\scs q$};
    \node at (-.3,.75) {$\scs p$};
\end{tikzpicture}}
&\text{if } i\cdot j<0.
\end{array}
\right.
\]

The algebra relations \eqref{RIII} are

\[
\hackcenter{\begin{tikzpicture}[scale=0.8]
    \draw[thick] (0,0) .. controls ++(0,1) and ++(0,-1) .. (1.2,2);
    \draw[thick] (.6,0) .. controls ++(0,.5) and ++(0,-.5) .. (0,1.0);
    \draw[thick] (0,1.0) .. controls ++(0,.5) and ++(0,-.5) .. (0.6,2);
    \draw[thick] (1.2,0) .. controls ++(0,1) and ++(0,-1) .. (0,2);
    \node at (1.2,-.2) {$\scs k$};
    \node at (.6,-.2) {$\scs j$};
    \node at (0,-.2) {$\scs i$};
\end{tikzpicture}}
\;\; - \;\;
\hackcenter{\begin{tikzpicture}[scale=0.8]
    \draw[thick] (0,0) .. controls ++(0,1) and ++(0,-1) .. (1.2,2);
    \draw[thick] (.6,0) .. controls ++(0,.5) and ++(0,-.5) .. (1.2,1.0);
    \draw[thick] (1.2,1.0) .. controls ++(0,.5) and ++(0,-.5) .. (0.6,2.0);
    \draw[thick] (1.2,0) .. controls ++(0,1) and ++(0,-1) .. (0,2.0);
    \node at (1.2,-.2) {$\scs k$};
    \node at (.6,-.2) {$\scs j$};
    \node at (0,-.2) {$\scs i$};
\end{tikzpicture}}
\;\; = \;\;
\left\{
\begin{array}{ll}
\displaystyle r(i)t(i,j)\sum_{a+b=-c(i,j)-1}
\hackcenter{\begin{tikzpicture}[scale=0.8]
    \draw[thick] (0,0) -- (0,2)  node[pos=.5, shape=coordinate](DOT){};
    \draw[thick] (.6,0) --  (.6,2) node[pos=.5, shape=coordinate](DOT2){};
    \draw[thick] (1.2,0) -- (1.2,2)  node[pos=.5, shape=coordinate](DOT1){};
    \filldraw  (DOT) circle (2.5pt);
    \filldraw  (DOT1) circle (2.5pt);
 \node at (-.4,1) {$\scs a$};
    \node at (1.6,1) {$\scs b$};
    \node at (1.2,-.2) {$\scs i$};
    \node at (.6,-.2) {$\scs j$};
    \node at (0,-.2) {$\scs i$};
\end{tikzpicture}}
\\
\displaystyle \qquad +r(i) \sum_{p,q}s^{pq}(i,j)\sum_{a+b=p-1}
\hackcenter{\begin{tikzpicture}[scale=0.8]
    \draw[thick] (0,0) -- (0,2)  node[pos=.5, shape=coordinate](DOT){};
    \draw[thick] (.6,0) --  (.6,2) node[pos=.5, shape=coordinate](DOT2){};
    \draw[thick] (1.2,0) -- (1.2,2)  node[pos=.5, shape=coordinate](DOT1){};
    \filldraw  (DOT) circle (2.5pt);
    \filldraw  (DOT1) circle (2.5pt);
    \filldraw  (DOT2) circle (2.5pt);
 \node at (-.4,1) {$\scs a$};
    \node at (1.6,1) {$\scs b$};
    \node at (.8,1.2) {$\scs q$};
    \node at (1.2,-.2) {$\scs i$};
    \node at (.6,-.2) {$\scs j$};
    \node at (0,-.2) {$\scs i$};
\end{tikzpicture}}
& \text{if } i=k, c(i,j)+1\leq 0
\\\\
0 & \text{otherwise.}
\end{array}
\right.
\]

\subsection{Example of type $A_1$}\label{typeA}
When $I=\{i\}$ and $n\geq 0$, Khovanov-Lauda-Rouquier algebra $R_n(Q_C)$ is isomorphic to the nil-Hecke ring $\mathrm{NH}_n$.

For each permutation $w\in \mathfrak{S}_n$ let $\psi_w = \psi_{j_0}\cdots\psi_{j_r}$, where $s_{j_0}\cdots s_{j_r}$ is a minimal presentation of $w$. This element does not depend on the choice of presentation.
The element
$$
e_{n}:=x_1^{n-1} x_2^{n-2}\cdots x_{n-1}\psi_{w_0}
$$
is a primitive idempotent in $\mathrm{NH}_n$, where $w_0$ is the longest element of $\mathfrak{S}_n$.
This element is diagrammatically described by
\begin{equation}
\label{idempotent}
\hackcenter{\begin{tikzpicture}[scale=0.8]
    \draw[thick] (0,0) -- (0,2.15);
    \draw[thick] (1.5,0) -- (1.5,2.15);
    \draw[thick] (3,0) -- (3,2.15);
    \draw[thick] (4.5,0) -- (4.5,2.15);
    \draw[thick] (6,0) -- (6,2.15);
    \node at (2.25,1.85) {$\cdots$};
    \node at (2.25,0.3) {$\cdots$};
    \fill[white] (-0.1,0.5) rectangle (6.1,1.5);
    \draw[thick] (-0.1,0.5) rectangle (6.1,1.5);
    \node at (3,1) {\Large $e_n$};
\end{tikzpicture}}
\quad :=
\hackcenter{\begin{tikzpicture}[scale=0.8]
    \draw[thick] (0,0) .. controls ++(0,1) and ++(0,-1) .. (6,2);
    \draw[thick] (6,0) .. controls ++(0,1) and ++(0,-1) .. (0,2) node[pos=1, shape=coordinate](DOT1){};
    \draw[thick] (4.5,0) .. controls ++(0,0.25) and ++(0,-0.25) .. (6,0.5);
    \draw[thick] (6,0.5) .. controls ++(0,0.75) and ++(0,-0.75) .. (1.5,2) node[pos=1, shape=coordinate](DOT2){};
    \draw[thick] (3,0) .. controls ++(0,0.5) and ++(0,-0.5) .. (6,1);
    \draw[thick] (6,1) .. controls ++(0,0.5) and ++(0,-0.5) .. (3,2) node[pos=1, shape=coordinate](DOT3){};
    \draw[thick] (1.5,0) .. controls ++(0,0.75) and ++(0,-0.75) .. (6,1.5);
    \draw[thick] (6,1.5) .. controls ++(0,0.25) and ++(0,-0.25) .. (4.5,2) node[pos=1, shape=coordinate](DOT4){};
    \draw[thick] (0,2) -- (0,2.25);
    \draw[thick] (1.5,2) -- (1.5,2.25);
    \draw[thick] (3,2) -- (3,2.25);
    \draw[thick] (4.5,2) -- (4.5,2.25);
    \draw[thick] (6,2) -- (6,2.25);
    \filldraw  (DOT1) circle (2.5pt);
    \filldraw  (DOT2) circle (2.5pt);
    \filldraw  (DOT3) circle (2.5pt);
    \filldraw  (DOT4) circle (2.5pt);
    \node at (2.25,2.15) {$\scs \cdots$};
    \node at (2.25,0.05) {$\scs \cdots$};
    \node at (-0.5,2.1) {$\scs n-1$};
    \node at (1.0,2.1) {$\scs n-2$};
    \node at (3.3,2.1) {$\scs 2$};
\end{tikzpicture}},
\end{equation}
where we omit the label $i$ on the strands.
\section{Khovanov-Lauda-Rouquier algebras of extended Cartan datum}

First, we introduce an extended Cartan datum $(\tilde{I},\cdot)$, and then we define Khovanov-Lauda-Rouquier algebra $R_n(\overline{Q}_{\tilde{C}})$, where $\tilde{C}$ is the associated matrix of the Cartan datum $(\tilde{I},\cdot)$ and $\overline{Q}_{\tilde{C}}$ is a specialization of scalar parameters $Q_{\tilde{C}}$.
In the next section, we define the Khovanov-Lauda-Rouquier subalgebra $S_{C}(\lambda,\nu)$ of $R_n(\overline{Q}_{\tilde{C}})$.

\subsection{Extended Cartan datum $(\tilde{I},\cdot)$}
\begin{definition}
The Cartan datum $(\tilde{I},\cdot)$ associated with $(I,\cdot)$ is defined by
\begin{itemize}
\item $\tilde{I}=I \cup \bar{I}$, where $\bar{I}=\{\,\bar{i}\, |\, i\in I\}$,
\item The bilinear form on $\Z[\tilde{I}]$: For $\tilde{i}$, $\tilde{j} \in \tilde{I}$,
\[
\tilde{i}\cdot \tilde{j}=\left\{
\begin{array}{ll}
\tilde{i}\cdot \tilde{j} &\text{if }\tilde{i}, \tilde{j} \in I\\[.5em]
-\frac{k\cdot k}{2} &\text{if }(\tilde{i},\tilde{j})=(k,\bar{k})\in I \times \bar{I} \text{ or }(\tilde{i},\tilde{j})=(\bar{k},k) \in \bar{I}\times I  \\[.5em]
k \cdot k& \text{if }\tilde{i}=\tilde{j}=\bar{k}\in \bar{I}\\[.5em]
0 & \text{otherwise.}
\end{array}
\right.
\]
\end{itemize}
Denote by $\tilde{C}$ the symmetrizable generalized Cartan matrix associated with $(\tilde{I},\cdot)$ and by $Q_{\tilde{C}}$ the scalar parameters, defined in Section \ref{cartan}.

\end{definition}

\subsection{Scalar parameters $\overline{Q}_{\tilde{C}}$}
For defining the algebra $S_{C}(\lambda,\nu)$, we use typical scalar parameters associated with Cartan datum $(\tilde{I},\cdot)$.
\begin{definition}
For Cartan datum $(\tilde{I},\cdot)$, scalar parameters
\begin{eqnarray*}
\overline{Q}_{\tilde{C}} &=&\{t(\tilde{i},\tilde{j})\in \Bbbk^{\times}\,|\,\tilde{i},\tilde{j}\in \tilde{I}\}
\cup
\{r(\tilde{i})\in \Bbbk^{\times}\,|\, \tilde{i}\in \tilde{I} \}
\\
&&\qquad
\cup
\{s^{pq}(\tilde{i},\tilde{j})\in\Bbbk\,|\,\tilde{i},\tilde{j}\in \tilde{I},\tilde{i}\not=\tilde{j}, p,q\in \Z_{> 0}, p (\tilde{i}\cdot \tilde{i})+q (\tilde{j}\cdot \tilde{j})=-2 (\tilde{i}\cdot \tilde{j})\}
\end{eqnarray*}
are defined by
\begin{eqnarray}
\label{scalar}
t(\tilde{i},\tilde{j})&=&\left\{\begin{array}{ll}
-1&\text{if}\quad(\tilde{i},\tilde{j})=(\bar{k},k) \in \bar{I}\times I \\[.5em]
1&\text{if}\quad(\tilde{i},\tilde{j})=(k,\bar{k}) \in I\times \bar{I} \\[.5em]
1&\text{if}\quad\tilde{i}\cdot\tilde{j}=0 \\[.5em]
t(\tilde{i},\tilde{j})&\text{otherwise}
\end{array}
\right.
\\
\nonumber
r(\tilde{i})&=&1\\
\nonumber
s^{pq}(\tilde{i},\tilde{j})&=&0.
\end{eqnarray}
\end{definition}

We discuss the subalgebra $S_{C}(\lambda,\nu)$ of Khovanov-Lauda-Rouquier algebra $R_n(\overline{Q}_{\tilde{C}})$ in the next section.

\subsection{Diagrammatic description of $R_n(\overline{Q}_{\tilde{C}})$}
We have a diagrammatic description of Khovanov-Lauda-Rouquier algebras as defined in Section \ref{diagram}.
The diagrammatic description for $R_n(\overline{Q}_{\tilde{C}})$ is the same in Section \ref{diagram}, except that we use solid lines and dashed lines in diagrams.
A strand labeled with an element of $I$ is drawn in solid and a strand labeled with an element of $\bar{I}$ is drawn in dashed.
That is, the generator $e(\mathbf{i})$ for $\mathbf{i}=(\tilde{i}_1,\tilde{i}_2,...,\tilde{i}_n) \in \tilde{I}^n$ is represented by the diagram consisting of solid and dashed vertical strands whose $j$-th strand is a solid strand labeled with $\tilde{i}_j\in I$ or a dashed strand labeled with $\tilde{i}_j\in \bar{I}$.

For instance, in the case $I=\{1,2,3\}$ and $\mathbf{i}=(1,2,\bar{3},2,3,\bar{1})$, we have
\[
e(\mathbf{i})=\hackcenter{\begin{tikzpicture}[scale=0.6]
    \node at (0,1.75) {};
    \draw[thick] (-1.5,0) -- (-1.5,1.5)  node[pos=.35, shape=coordinate](DOT){};
    \node at (-1.5,-.25) {$\scs  1$};
    \draw[thick] (-.75,0) -- (-.75,1.5)  node[pos=.35, shape=coordinate](DOT){};
    \node at (-.75,-.25) {$\scs  2$};
    \draw[densely dashed,thick] (0,0) -- (0,1.5)  node[pos=.35, shape=coordinate](DOT){};
    \node at (0,-.25) {$\scs  \bar{3}$};
    \draw[thick] (.75,0) -- (.75,1.5)  node[pos=.35, shape=coordinate](DOT){};
    \node at (.75,-.25) {$\scs  2$};
    \draw[thick] (1.5,0) -- (1.5,1.5)  node[pos=.35, shape=coordinate](DOT){};
    \node at (1.5,-.25) {$\scs  3$};
    \draw[densely dashed,thick] (2.25,0) -- (2.25,1.5)  node[pos=.35, shape=coordinate](DOT){};
    \node at (2.25,-.25) {$\scs  \bar{1}$};
\end{tikzpicture}}
\]

The generator $x_j e(\mathbf{i})$ is represented by a dot on the $j$-th strand as follows.
\begin{eqnarray*}
x_j e(\mathbf{i})&=&\left\{
\begin{array}{ll}
\hackcenter{\begin{tikzpicture}[scale=0.6]
    \node at (0,1.75) {};
    \draw[thick] (0,0) -- (0,1.5)  node[pos=.35, shape=coordinate](DOT){};
    \filldraw  (DOT) circle (2.5pt);
    \node at (0,-.35) {$\scs  \tilde{i}_j$};
    \draw[thick] (-1.5,0) -- (-1.5,1.5) ;
    \node at (-1.5,-.35) {$\scs  \tilde{i}_1$};
    \draw[thick] (1.5,0) -- (1.5,1.5) ;
    \node at (1.5,-.35) {$\scs  \tilde{i}_n$};
    \node at (-.75,.75) {$\cdots$};
    \node at (.75,.75) {$\cdots$};
\end{tikzpicture}}
& \text{if} \quad\tilde{i}_j\in I\\
\hackcenter{\begin{tikzpicture}[scale=0.6]
    \node at (0,1.75) {};
    \draw[densely dashed,thick] (0,0) -- (0,1.5)  node[pos=.35, shape=coordinate](DOT){};
    \filldraw  (DOT) circle (2.5pt);
    \node at (0,-.35) {$\scs  \tilde{i}_j$};
    \draw[thick] (-1.5,0) -- (-1.5,1.5) ;
    \node at (-1.5,-.35) {$\scs  \tilde{i}_1$};
    \draw[thick] (1.5,0) -- (1.5,1.5) ;
    \node at (1.5,-.35) {$\scs  \tilde{i}_n$};
    \node at (-.75,.75) {$\cdots$};
    \node at (.75,.75) {$\cdots$};
\end{tikzpicture}}
& \text{if} \quad\tilde{i}_j\in \bar{I}
\end{array}
\right.
\end{eqnarray*}

The generator $\psi_j e(\mathbf{i})$ is represented by a crossing of the $j$-th and ($j+1$)-th strands as follows.
\[
\psi_j e(\mathbf{i})=
\left\{
\begin{array}{ll}
\hackcenter{\begin{tikzpicture}[scale=0.6]
    \draw[thick] (0,0) .. controls (0,.75) and (1.5,.75) .. (1.5,1.5);
    \draw[thick] (1.5,0) .. controls (1.5,.75) and (0,.75) .. (0,1.5);
    \draw[thick] (-1.5,0) -- (-1.5,1.5) ;
    \draw[thick] (3,0) -- (3,1.5);
    \node at (-1.5,-.35) {$\scs  \tilde{i}_1$};
    \node at (0,-.35) {$\scs  \tilde{i}_j$};
    \node at (1.5,-.35) {$\scs  \tilde{i}_{j+1}$};
    \node at (3,-.35) {$\scs  \tilde{i}_n$};
    \node at (-.75,.75) {$\cdots$};
    \node at (2.25,.75) {$\cdots$};
\end{tikzpicture}}&\text{if}\quad \tilde{i}_j,\tilde{i}_{j+1}\in I\\[.5em]
\hackcenter{\begin{tikzpicture}[scale=0.6]
    \draw[thick] (0,0) .. controls (0,.75) and (1.5,.75) .. (1.5,1.5);
    \draw[densely dashed,thick] (1.5,0) .. controls (1.5,.75) and (0,.75) .. (0,1.5);
    \draw[thick] (-1.5,0) -- (-1.5,1.5) ;
    \draw[thick] (3,0) -- (3,1.5);
    \node at (-1.5,-.35) {$\scs  \tilde{i}_1$};
    \node at (0,-.35) {$\scs  \tilde{i}_j$};
    \node at (1.5,-.35) {$\scs  \tilde{i}_{j+1}$};
    \node at (3,-.35) {$\scs  \tilde{i}_n$};
    \node at (-.75,.75) {$\cdots$};
    \node at (2.25,.75) {$\cdots$};
\end{tikzpicture}}&\text{if}\quad \tilde{i}_j\in I,\tilde{i}_{j+1}\in \bar{I}\\[.5em]
\hackcenter{\begin{tikzpicture}[scale=0.6]
    \draw[densely dashed,thick] (0,0) .. controls (0,.75) and (1.5,.75) .. (1.5,1.5);
    \draw[thick] (1.5,0) .. controls (1.5,.75) and (0,.75) .. (0,1.5);
    \draw[thick] (-1.5,0) -- (-1.5,1.5) ;
    \draw[thick] (3,0) -- (3,1.5);
    \node at (-1.5,-.35) {$\scs  \tilde{i}_1$};
    \node at (0,-.35) {$\scs  \tilde{i}_j$};
    \node at (1.5,-.35) {$\scs  \tilde{i}_{j+1}$};
    \node at (3,-.35) {$\scs  \tilde{i}_n$};
    \node at (-.75,.75) {$\cdots$};
    \node at (2.25,.75) {$\cdots$};
\end{tikzpicture}}&\text{if}\quad \tilde{i}_j\in \bar{I},\tilde{i}_{j+1}\in I\\[.5em]
\hackcenter{\begin{tikzpicture}[scale=0.6]
    \draw[densely dashed,thick] (0,0) .. controls (0,.75) and (1.5,.75) .. (1.5,1.5);
    \draw[densely dashed,thick] (1.5,0) .. controls (1.5,.75) and (0,.75) .. (0,1.5);
    \draw[thick] (-1.5,0) -- (-1.5,1.5) ;
    \draw[thick] (3,0) -- (3,1.5);
    \node at (-1.5,-.35) {$\scs  \tilde{i}_1$};
    \node at (0,-.35) {$\scs  \tilde{i}_j$};
    \node at (1.5,-.35) {$\scs  \tilde{i}_{j+1}$};
    \node at (3,-.35) {$\scs  \tilde{i}_n$};
    \node at (-.75,.75) {$\cdots$};
    \node at (2.25,.75) {$\cdots$};
\end{tikzpicture}}&\text{if}\quad \tilde{i}_j,\tilde{i}_{j+1}\in \bar{I}
\end{array}
\right.
\]

\noindent
$\bullet$
The relations of \eqref{dot-slide1} and \eqref{dot-slide2} in $R_n(\overline{Q}_{\tilde{C}})$:

\[
\hackcenter{\begin{tikzpicture}[scale=0.8]
    \draw[thick] (0,0) .. controls (0,.75) and (.75,.75) .. (.75,1.5)
        node[pos=.25, shape=coordinate](DOT){};
    \draw[thick](.75,0) .. controls (.75,.75) and (0,.75) .. (0,1.5);
    \filldraw  (DOT) circle (2.5pt);
    \node at (.0,-.2) {$\scs i$};
    \node at (.75,-.2) {$\scs i$};
\end{tikzpicture}}
\quad -\quad
\hackcenter{\begin{tikzpicture}[scale=0.8]
    \draw[thick] (0,0) .. controls (0,.75) and (.75,.75) .. (.75,1.5)
        node[pos=.75, shape=coordinate](DOT){};
    \draw[thick](.75,0) .. controls (.75,.75) and (0,.75) .. (0,1.5);
    \filldraw  (DOT) circle (2.5pt);
    \node at (.0,-.2) {$\scs i$};
    \node at (.75,-.2) {$\scs i$};
\end{tikzpicture}}
\quad=\quad
\hackcenter{\begin{tikzpicture}[scale=0.8]
    \draw[thick] (0,0) --(0,1.5);
    \draw[thick](.75,0) -- (.75,1.5);
    \node at (.0,-.2) {$\scs i$};
    \node at (.75,-.2) {$\scs i$};
\end{tikzpicture}}
\quad=\quad
\hackcenter{\begin{tikzpicture}[scale=0.8]
    \draw[thick](0,0) .. controls (0,.75) and (.75,.75) .. (.75,1.5);
    \draw[thick] (.75,0) .. controls (.75,.75) and (0,.75) .. (0,1.5)
        node[pos=.75, shape=coordinate](DOT){};
    \filldraw  (DOT) circle (2.75pt);
    \node at (.0,-.2) {$\scs i$};
    \node at (.75,-.2) {$\scs i$};
\end{tikzpicture}}
\quad-\quad
\hackcenter{\begin{tikzpicture}[scale=0.8]
    \draw[thick](0,0) .. controls (0,.75) and (.75,.75) .. (.75,1.5);
    \draw[thick] (.75,0) .. controls (.75,.75) and (0,.75) .. (0,1.5)
        node[pos=.25, shape=coordinate](DOT){};
      \filldraw  (DOT) circle (2.75pt);
    \node at (.0,-.2) {$\scs i$};
    \node at (.75,-.2) {$\scs i$};
\end{tikzpicture}}
\]

\[
\hackcenter{\begin{tikzpicture}[scale=0.8]
    \draw[densely dashed,thick] (0,0) .. controls (0,.75) and (.75,.75) .. (.75,1.5)
        node[pos=.25, shape=coordinate](DOT){};
    \draw[densely dashed,thick](.75,0) .. controls (.75,.75) and (0,.75) .. (0,1.5);
    \filldraw  (DOT) circle (2.5pt);
    \node at (.0,-.2) {$\scs \bar{i}$};
    \node at (.75,-.2) {$\scs \bar{i}$};
\end{tikzpicture}}
\quad -\quad
\hackcenter{\begin{tikzpicture}[scale=0.8]
    \draw[densely dashed,thick] (0,0) .. controls (0,.75) and (.75,.75) .. (.75,1.5)
        node[pos=.75, shape=coordinate](DOT){};
    \draw[densely dashed,thick](.75,0) .. controls (.75,.75) and (0,.75) .. (0,1.5);
    \filldraw  (DOT) circle (2.5pt);
    \node at (.0,-.2) {$\scs \bar{i}$};
    \node at (.75,-.2) {$\scs \bar{i}$};
\end{tikzpicture}}
\quad=\quad
\hackcenter{\begin{tikzpicture}[scale=0.8]
    \draw[densely dashed,thick] (0,0) --(0,1.5);
    \draw[densely dashed,thick](.75,0) -- (.75,1.5);
    \node at (.0,-.2) {$\scs \bar{i}$};
    \node at (.75,-.2) {$\scs \bar{i}$};
\end{tikzpicture}}
\quad=\quad
\hackcenter{\begin{tikzpicture}[scale=0.8]
    \draw[densely dashed,thick](0,0) .. controls (0,.75) and (.75,.75) .. (.75,1.5);
    \draw[densely dashed,thick] (.75,0) .. controls (.75,.75) and (0,.75) .. (0,1.5)
        node[pos=.75, shape=coordinate](DOT){};
    \filldraw  (DOT) circle (2.75pt);
    \node at (.0,-.2) {$\scs \bar{i}$};
    \node at (.75,-.2) {$\scs \bar{i}$};
\end{tikzpicture}}
\quad-\quad
\hackcenter{\begin{tikzpicture}[scale=0.8]
    \draw[densely dashed,thick](0,0) .. controls (0,.75) and (.75,.75) .. (.75,1.5);
    \draw[densely dashed,thick] (.75,0) .. controls (.75,.75) and (0,.75) .. (0,1.5)
        node[pos=.25, shape=coordinate](DOT){};
      \filldraw  (DOT) circle (2.75pt);
    \node at (.0,-.2) {$\scs \bar{i}$};
    \node at (.75,-.2) {$\scs \bar{i}$};
\end{tikzpicture}}
\]

\[
\hackcenter{\begin{tikzpicture}[scale=0.8]
    \draw[thick] (0,0) .. controls (0,.75) and (.75,.75) .. (.75,1.5)
        node[pos=.25, shape=coordinate](DOT){};
    \draw[thick](.75,0) .. controls (.75,.75) and (0,.75) .. (0,1.5);
    \filldraw  (DOT) circle (2.5pt);
    \node at (.0,-.2) {$\scs \tilde{i}$};
    \node at (.75,-.2) {$\scs \tilde{j}$};
\end{tikzpicture}}
\quad =\quad
\hackcenter{\begin{tikzpicture}[scale=0.8]
    \draw[thick] (0,0) .. controls (0,.75) and (.75,.75) .. (.75,1.5)
        node[pos=.75, shape=coordinate](DOT){};
    \draw[thick](.75,0) .. controls (.75,.75) and (0,.75) .. (0,1.5);
    \filldraw  (DOT) circle (2.5pt);
    \node at (.0,-.2) {$\scs \tilde{i}$};
    \node at (.75,-.2) {$\scs \tilde{j}$};
\end{tikzpicture}}
\qquad \quad
\hackcenter{\begin{tikzpicture}[scale=0.8]
    \draw[thick](0,0) .. controls (0,.75) and (.75,.75) .. (.75,1.5);
    \draw[thick] (.75,0) .. controls (.75,.75) and (0,.75) .. (0,1.5)
        node[pos=.75, shape=coordinate](DOT){};
    \filldraw  (DOT) circle (2.75pt);
    \node at (.0,-.2) {$\scs \tilde{i}$};
    \node at (.75,-.2) {$\scs \tilde{j}$};
\end{tikzpicture}}
\quad=\quad
\hackcenter{\begin{tikzpicture}[scale=0.8]
    \draw[thick](0,0) .. controls (0,.75) and (.75,.75) .. (.75,1.5);
    \draw[thick] (.75,0) .. controls (.75,.75) and (0,.75) .. (0,1.5)
        node[pos=.25, shape=coordinate](DOT){};
      \filldraw  (DOT) circle (2.75pt);
    \node at (.0,-.2) {$\scs \tilde{i}$};
    \node at (.75,-.2) {$\scs \tilde{j}$};
\end{tikzpicture}}
\qquad \text{for } \quad\tilde{i}\not=\tilde{j},
\]
where the type of lines in the above diagrams is solid or dashed depends on labels $\tilde{i}$ and $\tilde{j}$. 
The type of lines is omitted in the above last two diagrams.

\noindent 
$\bullet$
The relations of \eqref{RII} in $R_n(\overline{Q}_{\tilde{C}})$:
\begin{equation}\label{RII-diag}
\hackcenter{\begin{tikzpicture}[scale=0.8]
    \draw[thick] (0,0) .. controls ++(0,.5) and ++(0,-.5) .. (.75,.75);
    \draw[thick] (.75,0) .. controls ++(0,.5) and ++(0,-.5) .. (0,.75);
    \draw[thick] (0,.75) .. controls ++(0,.5) and ++(0,-.5) .. (.75,1.5);
    \draw[thick] (.75,.75) .. controls ++(0,.5) and ++(0,-.5) .. (0,1.5);
    \node at (0,-.2) {$\scs i$};
    \node at (.8,-.2) {$\scs j$};
\end{tikzpicture}}
 \;\; = \;\;
\left\{
\begin{array}{ll}
0&\text{if } i=j\\[1em]
\hackcenter{\begin{tikzpicture}[scale=0.8]
    \draw[thick] (0,0) -- (0,1.5);
    \draw[thick] (.75,0) -- (.75,1.5);
    \node at (0,-.2) {$\scs i$};
    \node at (.8,-.2) {$\scs j$};
    \end{tikzpicture}}
&\text{if } i\cdot j=0\\[1em]
t(i,j)
\txt{
\hackcenter{\begin{tikzpicture}[scale=0.8]
    \draw[thick] (0,0) -- (0,1.5)  node[pos=.55, shape=coordinate](DOT){};
    \filldraw  (DOT) circle (2.5pt);
    \draw[thick] (.75,0) -- (.75,1.5);
    \node at (0,-.2) {$\scs i$};
    \node at (.8,-.2) {$\scs j$};
    \node at (-.7,.75) {$\scs -c(i,j)$};
\end{tikzpicture}}
$+t(j,i)$
\hackcenter{\begin{tikzpicture}[scale=0.8]
    \draw[thick] (0,0) -- (0,1.5);
    \draw[thick] (.75,0) -- (.75,1.5) node[pos=.55, shape=coordinate](DOT2){};
    \filldraw  (DOT2) circle (2.75pt);
    \node at (0,-.2) {$\scs i$};
    \node at (.8,-.2) {$\scs j$};
    \node at (1.5,.75) {$\scs -c(j,i)$};
\end{tikzpicture}}
}
&\text{if } i\cdot j<0
\end{array}
\right.
\end{equation}

\begin{equation}\label{R2-thin-thick}
\hackcenter{\begin{tikzpicture}[scale=0.8]
    \draw[densely dashed,thick] (0,0) .. controls ++(0,.5) and ++(0,-.5) .. (.75,.75);
    \draw[thick] (.75,0) .. controls ++(0,.5) and ++(0,-.5) .. (0,.75);
    \draw[thick] (0,.75) .. controls ++(0,.5) and ++(0,-.5) .. (.75,1.5);
    \draw[densely dashed,thick] (.75,.75) .. controls ++(0,.5) and ++(0,-.5) .. (0,1.5);
    \node at (0,-.2) {$\scs \bar{i}$};
    \node at (.8,-.2) {$\scs j$};
\end{tikzpicture}}
 \;\; = \;\;
\left\{
\begin{array}{ll}
\text{$-$
\hackcenter{\begin{tikzpicture}[scale=0.8]
    \draw[densely dashed,thick] (0,0) -- (0,1.5)  node[pos=.55, shape=coordinate](DOT){};
    \filldraw  (DOT) circle (2.5pt);
    \draw[thick] (.75,0) -- (.75,1.5);
    \node at (0,-.2) {$\scs \bar{i}$};
    \node at (.8,-.2) {$\scs i$};
\end{tikzpicture}}
$+$
\hackcenter{\begin{tikzpicture}[scale=0.8]
    \draw[densely dashed,thick] (0,0) -- (0,1.5);
    \draw[thick] (.75,0) -- (.75,1.5) node[pos=.55, shape=coordinate](DOT2){};
    \filldraw  (DOT2) circle (2.75pt);
    \node at (0,-.2) {$\scs \bar{i}$};
    \node at (.8,-.2) {$\scs i$};
\end{tikzpicture}}
}
&\text{if } i=j\\[1em]
\hackcenter{\begin{tikzpicture}[scale=0.8]
    \draw[densely dashed,thick] (0,0) -- (0,1.5);
    \draw[thick] (.75,0) -- (.75,1.5);
    \node at (0,-.2) {$\scs \bar{i}$};
    \node at (.8,-.2) {$\scs j$};
    \end{tikzpicture}}
&\text{otherwise}
\end{array}
\right.
\hackcenter{\begin{tikzpicture}[scale=0.8]
    \draw[thick] (0,0) .. controls ++(0,.5) and ++(0,-.5) .. (.75,.75);
    \draw[densely dashed,thick] (.75,0) .. controls ++(0,.5) and ++(0,-.5) .. (0,.75);
    \draw[densely dashed,thick] (0,.75) .. controls ++(0,.5) and ++(0,-.5) .. (.75,1.5);
    \draw[thick] (.75,.75) .. controls ++(0,.5) and ++(0,-.5) .. (0,1.5);
    \node at (0,-.2) {$\scs i$};
    \node at (.8,-.2) {$\scs \bar{j}$};
\end{tikzpicture}}
 \;\; = \;\;
\left\{
\begin{array}{ll}
\text{\hackcenter{\begin{tikzpicture}[scale=0.8]
    \draw[thick] (0,0) -- (0,1.5)  node[pos=.55, shape=coordinate](DOT){};
    \filldraw  (DOT) circle (2.5pt);
    \draw[densely dashed,thick] (.75,0) -- (.75,1.5);
    \node at (0,-.2) {$\scs i$};
    \node at (.8,-.2) {$\scs \bar{i}$};
\end{tikzpicture}}
$-$
\hackcenter{\begin{tikzpicture}[scale=0.8]
    \draw[thick] (0,0) -- (0,1.5);
    \draw[densely dashed,thick] (.75,0) -- (.75,1.5) node[pos=.55, shape=coordinate](DOT2){};
    \filldraw  (DOT2) circle (2.75pt);
    \node at (0,-.2) {$\scs i$};
    \node at (.8,-.2) {$\scs \bar{i}$};
\end{tikzpicture}}
}
&\text{if } i=j\\[1em]
\hackcenter{\begin{tikzpicture}[scale=0.8]
    \draw[thick] (0,0) -- (0,1.5);
    \draw[densely dashed,thick] (.75,0) -- (.75,1.5);
    \node at (0,-.2) {$\scs i$};
    \node at (.8,-.2) {$\scs \bar{j}$};
    \end{tikzpicture}}
&\text{otherwise}
\end{array}
\right.
\end{equation}

\[
\hackcenter{\begin{tikzpicture}[scale=0.8]
    \draw[densely dashed,thick] (0,0) .. controls ++(0,.5) and ++(0,-.5) .. (.75,.75);
    \draw[densely dashed,thick] (.75,0) .. controls ++(0,.5) and ++(0,-.5) .. (0,.75);
    \draw[densely dashed,thick] (0,.75) .. controls ++(0,.5) and ++(0,-.5) .. (.75,1.5);
    \draw[densely dashed,thick] (.75,.75) .. controls ++(0,.5) and ++(0,-.5) .. (0,1.5);
    \node at (0,-.2) {$\scs \bar{i}$};
    \node at (.8,-.2) {$\scs \bar{j}$};
\end{tikzpicture}}
 \;\; = \;\;
\left\{
\begin{array}{ll}
0&\text{if } i=j\\[1em]
\hackcenter{\begin{tikzpicture}[scale=0.8]
    \draw[densely dashed,thick] (0,0) -- (0,1.5);
    \draw[densely dashed,thick] (.75,0) -- (.75,1.5);
    \node at (0,-.2) {$\scs \bar{i}$};
    \node at (.8,-.2) {$\scs \bar{j}$};
    \end{tikzpicture}}
&\text{otherwise}
\end{array}
\right.
\]

\noindent
$\bullet$
The relations of \eqref{RIII} in $R_n(\overline{Q}_{\tilde{C}})$:

Unless $\tilde{i}=\tilde{k}$ and $\tilde{i}\cdot\tilde{j}<0$, the following relation holds.
\begin{equation}\label{R3-gene}
\hackcenter{\begin{tikzpicture}[scale=0.8]
    \draw[thick] (0,0) .. controls ++(0,1) and ++(0,-1) .. (1.2,2);
    \draw[thick] (.6,0) .. controls ++(0,.5) and ++(0,-.5) .. (0,1.0);
    \draw[thick] (0,1.0) .. controls ++(0,.5) and ++(0,-.5) .. (0.6,2);
    \draw[thick] (1.2,0) .. controls ++(0,1) and ++(0,-1) .. (0,2);
    \node at (1.2,-.2) {$\scs \tilde{k}$};
    \node at (.6,-.2) {$\scs \tilde{j}$};
    \node at (0,-.2) {$\scs \tilde{i}$};
\end{tikzpicture}}
\;\; - \;\;
\hackcenter{\begin{tikzpicture}[scale=0.8]
    \draw[thick] (0,0) .. controls ++(0,1) and ++(0,-1) .. (1.2,2);
    \draw[thick] (.6,0) .. controls ++(0,.5) and ++(0,-.5) .. (1.2,1.0);
    \draw[thick] (1.2,1.0) .. controls ++(0,.5) and ++(0,-.5) .. (0.6,2.0);
    \draw[thick] (1.2,0) .. controls ++(0,1) and ++(0,-1) .. (0,2.0);
    \node at (1.2,-.2) {$\scs \tilde{k}$};
    \node at (.6,-.2) {$\scs \tilde{j}$};
    \node at (0,-.2) {$\scs \tilde{i}$};
\end{tikzpicture}}
\;\; = \;\;
0,
\end{equation}
where the type of lines in the above diagrams is also omitted.

Otherwise, we have
\begin{equation}\label{RIII-diag}
\hackcenter{\begin{tikzpicture}[scale=0.8]
    \draw[thick] (0,0) .. controls ++(0,1) and ++(0,-1) .. (1.2,2);
    \draw[thick] (.6,0) .. controls ++(0,.5) and ++(0,-.5) .. (0,1.0);
    \draw[thick] (0,1.0) .. controls ++(0,.5) and ++(0,-.5) .. (0.6,2);
    \draw[thick] (1.2,0) .. controls ++(0,1) and ++(0,-1) .. (0,2);
    \node at (1.2,-.2) {$\scs i$};
    \node at (.6,-.2) {$\scs j$};
    \node at (0,-.2) {$\scs i$};
\end{tikzpicture}}
\;\; - \;\;
\hackcenter{\begin{tikzpicture}[scale=0.8]
    \draw[thick] (0,0) .. controls ++(0,1) and ++(0,-1) .. (1.2,2);
    \draw[thick] (.6,0) .. controls ++(0,.5) and ++(0,-.5) .. (1.2,1.0);
    \draw[thick] (1.2,1.0) .. controls ++(0,.5) and ++(0,-.5) .. (0.6,2.0);
    \draw[thick] (1.2,0) .. controls ++(0,1) and ++(0,-1) .. (0,2.0);
    \node at (1.2,-.2) {$\scs i$};
    \node at (.6,-.2) {$\scs j$};
    \node at (0,-.2) {$\scs i$};
\end{tikzpicture}}
\;\; = \;\;
t(i,j) \sum_{a+b=-c(i,j)-1}
\hackcenter{\begin{tikzpicture}[scale=0.8]
    \draw[thick] (0,0) -- (0,2)  node[pos=.5, shape=coordinate](DOT){};
    \draw[thick] (.6,0) --  (.6,2) node[pos=.5, shape=coordinate](DOT2){};
    \draw[thick] (1.2,0) -- (1.2,2)  node[pos=.5, shape=coordinate](DOT1){};
    \filldraw  (DOT) circle (2.5pt);
    \filldraw  (DOT1) circle (2.5pt);
 \node at (-.4,1) {$\scs a$};
    \node at (1.6,1) {$\scs b$};
    \node at (1.2,-.2) {$\scs i$};
    \node at (.6,-.2) {$\scs j$};
    \node at (0,-.2) {$\scs i$};
\end{tikzpicture}}
\end{equation}

\[
\hackcenter{\begin{tikzpicture}[scale=0.8]
    \draw[thick] (0,0) .. controls ++(0,1) and ++(0,-1) .. (1.2,2);
    \draw[densely dashed,thick] (.6,0) .. controls ++(0,.5) and ++(0,-.5) .. (0,1.0);
    \draw[densely dashed,thick] (0,1.0) .. controls ++(0,.5) and ++(0,-.5) .. (0.6,2);
    \draw[thick] (1.2,0) .. controls ++(0,1) and ++(0,-1) .. (0,2);
    \node at (1.2,-.2) {$\scs i$};
    \node at (.6,-.2) {$\scs \bar{i}$};
    \node at (0,-.2) {$\scs i$};
\end{tikzpicture}}
\;\; - \;\;
\hackcenter{\begin{tikzpicture}[scale=0.8]
    \draw[thick] (0,0) .. controls ++(0,1) and ++(0,-1) .. (1.2,2);
    \draw[densely dashed,thick] (.6,0) .. controls ++(0,.5) and ++(0,-.5) .. (1.2,1.0);
    \draw[densely dashed,thick] (1.2,1.0) .. controls ++(0,.5) and ++(0,-.5) .. (0.6,2.0);
    \draw[thick] (1.2,0) .. controls ++(0,1) and ++(0,-1) .. (0,2.0);
    \node at (1.2,-.2) {$\scs i$};
    \node at (.6,-.2) {$\scs \bar{i}$};
    \node at (0,-.2) {$\scs i$};
\end{tikzpicture}}
\;\; = \;\;
\hackcenter{\begin{tikzpicture}[scale=0.8]
    \draw[thick] (0,0) -- (0,2)  node[pos=.5, shape=coordinate](DOT){};
    \draw[densely dashed,thick] (.6,0) --  (.6,2) node[pos=.5, shape=coordinate](DOT2){};
    \draw[thick] (1.2,0) -- (1.2,2)  node[pos=.5, shape=coordinate](DOT1){};
    \node at (1.2,-.2) {$\scs i$};
    \node at (.6,-.2) {$\scs \bar{i}$};
    \node at (0,-.2) {$\scs i$};
\end{tikzpicture}}
\qquad\qquad
\hackcenter{\begin{tikzpicture}[scale=0.8]
    \draw[densely dashed,thick] (0,0) .. controls ++(0,1) and ++(0,-1) .. (1.2,2);
    \draw[thick] (.6,0) .. controls ++(0,.5) and ++(0,-.5) .. (0,1.0);
    \draw[thick] (0,1.0) .. controls ++(0,.5) and ++(0,-.5) .. (0.6,2);
    \draw[densely dashed,thick] (1.2,0) .. controls ++(0,1) and ++(0,-1) .. (0,2);
    \node at (1.2,-.2) {$\scs \bar{i}$};
    \node at (.6,-.2) {$\scs i$};
    \node at (0,-.2) {$\scs \bar{i}$};
\end{tikzpicture}}
\;\; - \;\;
\hackcenter{\begin{tikzpicture}[scale=0.8]
    \draw[densely dashed,thick] (0,0) .. controls ++(0,1) and ++(0,-1) .. (1.2,2);
    \draw[thick] (.6,0) .. controls ++(0,.5) and ++(0,-.5) .. (1.2,1.0);
    \draw[thick] (1.2,1.0) .. controls ++(0,.5) and ++(0,-.5) .. (0.6,2.0);
    \draw[densely dashed,thick] (1.2,0) .. controls ++(0,1) and ++(0,-1) .. (0,2.0);
    \node at (1.2,-.2) {$\scs \bar{i}$};
    \node at (.6,-.2) {$\scs i$};
    \node at (0,-.2) {$\scs \bar{i}$};
\end{tikzpicture}}
\;\; = \;\;
-
\hackcenter{\begin{tikzpicture}[scale=0.8]
    \draw[densely dashed,thick] (0,0) -- (0,2)  node[pos=.5, shape=coordinate](DOT){};
    \draw[thick] (.6,0) --  (.6,2) node[pos=.5, shape=coordinate](DOT2){};
    \draw[densely dashed,thick] (1.2,0) -- (1.2,2)  node[pos=.5, shape=coordinate](DOT1){};
    \node at (1.2,-.2) {$\scs \bar{i}$};
    \node at (.6,-.2) {$\scs i$};
    \node at (0,-.2) {$\scs \bar{i}$};
\end{tikzpicture}}.
\]

\section{Khovanov-Lauda-Rouquier subalgebra $S_{C}(\lambda,\nu)$}
We define algebra $S_{C}(\lambda,\nu)$ associated with $\lambda \in (\Z_{\geq 0}[\bar{I}])^m$ and $\nu\in I^{\ell}$ as a subalgebra of the Khovanov-Lauda-Rouquier algebra of extended Cartan datum $(\tilde{I},\cdot)$. 
\subsection{Thick dashed strand for $e(\mathbf{i})$}
Let $\lambda=(\lambda_1,...,\lambda_{m})$ be an element in $(\Z_{\geq 0}[\bar{I}])^m$, and let $\nu=(\nu_1,...,\nu_\ell)$ be an element in $I^{\ell}$.
We write $|\lambda|:= \sum_{k=1}^m\sum_{i\in I}\lambda_k^{(i)}$, where $\lambda_k=\sum_{i\in I}\lambda_k^{(i)} \bar{i}\in \Z_{\geq 0}[\bar{I}]$.

The following form of idempotent $e_{\lambda_k}$ is primitive since we have $\overline{i}\cdot\overline{j}=0$ for $\overline{i}\not=\overline{j}$
\begin{equation}
e_{\lambda_k}=
\hackcenter{\begin{tikzpicture}[scale=0.8]
    \draw[densely dashed, thick, double] (0,0) -- (0,2.15);
    \node at (0,-.25) {$\scs \lambda_k$};
\end{tikzpicture}}
\quad :=
\hackcenter{\begin{tikzpicture}[scale=0.8]
    \draw[densely dashed,thick] (0,0) -- (0,2.15);
    \draw[densely dashed,thick] (1.5,0) -- (1.5,2.15);
    \node at (0,-.25) {$\scs \overline{1}$};
    \node at (1.5,-.25) {$\scs \overline{1}$};
    \node at (0.75,1.85) {$\cdots$};
    \node at (0.75,0.25) {$\cdots$};
    \fill[white] (-0.1,0.5) rectangle (1.6,1.5);
    \draw[thick] (-0.1,0.5) rectangle (1.6,1.5);
    \node at (.75,1) {\Large $e_{\lambda_k^{(1)}}$};
    \draw[densely dashed,thick] (2,0) -- (2,2.15);
    \draw[densely dashed,thick] (3.5,0) -- (3.5,2.15);
    \node at (2,-.25) {$\scs \overline{2}$};
    \node at (3.5,-.25) {$\scs \overline{2}$};
    \node at (2.75,1.85) {$\cdots$};
    \node at (2.75,0.25) {$\cdots$};
    \fill[white] (1.9,0.5) rectangle (3.6,1.5);
    \draw[thick] (1.9,0.5) rectangle (3.6,1.5);
    \node at (2.75,1) {\Large $e_{\lambda_k^{(2)}}$};
    \draw[densely dashed,thick] (4.5,0) -- (4.5,2.15);
    \draw[densely dashed,thick] (6,0) -- (6,2.15);
    \node at (4.5,-.25) {$\scs \overline{N}$};
    \node at (6,-.25) {$\scs \overline{N}$};
    \node at (5.25,1.85) {$\cdots$};
    \node at (5.25,0.25) {$\cdots$};
    \fill[white] (4.4,0.5) rectangle (6.1,1.5);
    \draw[thick] (4.4,0.5) rectangle (6.1,1.5);
    \node at (5.25,1) {\Large $e_{\lambda_k^{(N)}}$};
    \node at (4,1) {$\cdots$};
\end{tikzpicture}}
\end{equation}
where the box diagrams are introduced in Section \ref{typeA} (See Eq. \eqref{idempotent}).

We define $\mathrm{Seq}(\lambda,\nu)$ by the set of all sequences $\mathbf{i}=(\mathbf{i}_1,...,\mathbf{i}_{m+\ell})$ in which 
\begin{itemize}
\item $n$ entries are elements in $\nu$ appearing exactly once and 
\item $m$ entries are elements in $\lambda$ appearing exactly once and in the order of the sequence $\lambda$.
\end{itemize}

The idempotent $e(\mathbf{i})$ for $\mathbf{i}\in\mathrm{Seq}(\lambda,\nu)$ is represented by the diagram consisting of solid and thick dashed vertical strands whose $j$-th strand is the solid strand labeled with $\mathbf{i}_j\in I$ or the thick dashed strand labeled with $\mathbf{i}_j\in \Z_{\geq 0}[\bar{I}]$.

For instance, when $\lambda=(\lambda_1,\lambda_2)$, $\nu=(\nu_1,\nu_2,\nu_3)$ and $\mathbf{i}=(\nu_3,\lambda_1,\lambda_2,\nu_1,\nu_2)\in \mathrm{Seq}(\lambda,\nu)$, the element $e(\mathbf{i})$ is an idempotent of Khovanov-Lauda-Rouquier algebra $R_{|\lambda|+3}(\overline{Q}_{\tilde{C}})$
\begin{equation}
e(\mathbf{i})=
\hackcenter{\begin{tikzpicture}[scale=0.8]
    \draw[thick] (0,0) -- (0,2.15);
    \node at (0,-.25) {$\scs \nu_3$};
    \draw[densely dashed, thick, double] (1,0) -- (1,2.15);
    \node at (1,-.25) {$\scs \lambda_1$};
    \draw[densely dashed, thick, double] (2,0) -- (2,2.15);
    \node at (2,-.25) {$\scs \lambda_2$};
    \draw[thick] (3,0) -- (3,2.15);
    \node at (3,-.25) {$\scs \nu_1$};
    \draw[thick] (4,0) -- (4,2.15);
    \node at (4,-.25) {$\scs \nu_2$};
\end{tikzpicture}}.
\end{equation}

\subsection{Dots on solid strands}
For an idempotent $e(\mathbf{i})\in \mathrm{Seq}(\lambda, \nu)$, we define $y_j$ ($1\leq j\leq m+\ell$) by 
\begin{equation}
y_je(\mathbf{i})=\left\{
\begin{array}{ll}
x_{\alpha}e(\mathbf{i})&\text{if }\mathbf{i}_j\in I,\\
0&\text{if }\mathbf{i}_j\in  \Z_{\geq 0}[\overline{I}],
\end{array}
\right.
\end{equation}
where $\alpha$ is the position number of $\mathbf{i}_j$ counting from the left when we regard $e(\mathbf{i})$ as the idempotent in $R_{|\lambda|+\ell}(\overline{Q}_{\tilde{C}})$.

For instance, when $\lambda=(\lambda_1,\lambda_2)$, $\nu=(\nu_1,\nu_2,\nu_3)$ and $\mathbf{i}=(\nu_3,\lambda_1,\lambda_2,\nu_1,\nu_2)\in \mathrm{Seq}(\lambda,\nu)$, the element $y_j$ ($j=1,2,3,4,5$) is the following generator of the Khovanov-Lauda-Rouquier algebra $R_{|\lambda|+3}(\overline{Q}_{\tilde{C}})$
\begin{equation}
\nonumber
y_1 e(\mathbf{i})=x_1 e(\mathbf{i}),\quad
y_2 e(\mathbf{i})=0,\quad
y_3 e(\mathbf{i})=0,\quad
y_4 e(\mathbf{i})=x_{|\lambda+_1|+|\lambda_2|+2} e(\mathbf{i}),\quad
y_5 e(\mathbf{i})=x_{|\lambda+_1|+|\lambda_2|+3} e(\mathbf{i}).
\end{equation}
The diagrammatic description of these generators is as follows.
\begin{equation}
\nonumber
y_1e(\mathbf{i})=
\hackcenter{\begin{tikzpicture}[scale=0.8]
    \draw[thick] (0,0) -- (0,2.15) node[pos=.5, shape=coordinate](DOT){};
    \node at (0,-.25) {$\scs \nu_3$};
    \draw[densely dashed, thick, double] (1,0) -- (1,2.15);
    \node at (1,-.25) {$\scs \lambda_1$};
    \draw[densely dashed, thick, double] (2,0) -- (2,2.15);
    \node at (2,-.25) {$\scs \lambda_2$};
    \draw[thick] (3,0) -- (3,2.15);
    \node at (3,-.25) {$\scs \nu_1$};
    \draw[thick] (4,0) -- (4,2.15);
    \node at (4,-.25) {$\scs \nu_2$};
   \filldraw  (DOT) circle (2.5pt);
\end{tikzpicture}}\quad
y_4e(\mathbf{i})=
\hackcenter{\begin{tikzpicture}[scale=0.8]
    \draw[thick] (0,0) -- (0,2.15);
    \node at (0,-.25) {$\scs \nu_3$};
    \draw[densely dashed, thick, double] (1,0) -- (1,2.15);
    \node at (1,-.25) {$\scs \lambda_1$};
    \draw[densely dashed, thick, double] (2,0) -- (2,2.15);
    \node at (2,-.25) {$\scs \lambda_2$};
    \draw[thick] (3,0) -- (3,2.15) node[pos=.5, shape=coordinate](DOT){};
    \node at (3,-.25) {$\scs \nu_1$};
    \draw[thick] (4,0) -- (4,2.15);
    \node at (4,-.25) {$\scs \nu_2$};
   \filldraw  (DOT) circle (2.5pt);
\end{tikzpicture}}\quad
y_5e(\mathbf{i})=
\hackcenter{\begin{tikzpicture}[scale=0.8]
    \draw[thick] (0,0) -- (0,2.15);
    \node at (0,-.25) {$\scs \nu_3$};
    \draw[densely dashed, thick, double] (1,0) -- (1,2.15);
    \node at (1,-.25) {$\scs \lambda_1$};
    \draw[densely dashed, thick, double] (2,0) -- (2,2.15);
    \node at (2,-.25) {$\scs \lambda_2$};
    \draw[thick] (3,0) -- (3,2.15);
    \node at (3,-.25) {$\scs \nu_1$};
    \draw[thick] (4,0) -- (4,2.15) node[pos=.5, shape=coordinate](DOT){};
    \node at (4,-.25) {$\scs \nu_2$};
   \filldraw  (DOT) circle (2.5pt);
\end{tikzpicture}}
\end{equation}
Remark that we do not have the dot for $y_je(\mathbf{i})$ if $\mathbf{i}_j\in  \Z_{\geq 0}[\overline{I}]$ since $y_je(\mathbf{i})=0$.

\subsection{Dots on thick dashed strand}
Let $\mathfrak{S}_{\lambda_k}$ ($\lambda_k=\sum_{i\in I}\lambda_k^{(i)} \bar{i}\in \Z_{\geq 0}[\overline{I}]$) be the direct product of symmetric groups $\mathfrak{S}_{\lambda_k^{(1)}}\times \mathfrak{S}_{\lambda_k^{(2)}}\times \cdots \times \mathfrak{S}_{\lambda_k^{(N)}}$.
The group $\mathfrak{S}_{\lambda_k}$ is a subgroup of the symmetric group $\mathfrak{S}_{\sum_{i\in I} \lambda_{k}^{(i)}}$. Therefore, we naturally have the action of $\mathfrak{S}_{\lambda_k}$ on the polynomial ring
$$
\Bbbk[X_{\lambda_k}]=\Bbbk[x_\ell|1\leq \ell\leq \sum_{i\in I} \lambda_{k}^{(i)}].
$$
and we have the invariant ring $\Bbbk[X_{\lambda_k}]^{\mathfrak{S}_{\lambda_k}}$. 
This ring has a generating set composed of the elementary symmetric functions $E_{\lambda_k}=\{E_{k,j}^{(i)} | i\in I, 1\leq j\leq \lambda_k^{(i)}\}$, where $E_{k,j}^{(i)}$ is the $j$-th symmetric function in the $i$-th block $\Bbbk[x_{\ell} | \sum_{a=1}^{i-1}\lambda_k^{(a)}+1\leq \ell\leq \sum_{a=1}^{i}\lambda_k^{(a)}]$.
We simply write
$$
\Bbbk[E_{\lambda_k}]=\Bbbk[X_{\lambda_k}]^{\mathfrak{S}_{\lambda_k}}.
$$

The element $f\in \Bbbk[E_{\lambda_k}]$ satisfies the commutativity condition $f\cdot e_{\lambda_k}=e_{\lambda_k}\cdot f$.  We represent the diagrammatic description of $f$ as the dot labeled with $f$ on the thick dashed strand of $e_{\lambda_k}$:
\begin{equation}
\hackcenter{\begin{tikzpicture}[scale=0.8]
    \draw[densely dashed, thick, double] (0,0) -- (0,2.15)  node[pos=.5, shape=coordinate](DOT){};
    \node at (0,-.25) {$\scs \lambda_k$};
    \filldraw  (DOT) circle (3pt);
    \node at (.25,1) {$\scs f$};
\end{tikzpicture}}
\end{equation}

For the diagram description of generators in $\Bbbk[E_{\lambda_k}]$, we omit $k$ in the letter $E_{k,j}^{(i)}$ as below since the strand has the information of $k$:
\begin{equation}
\hackcenter{\begin{tikzpicture}[scale=0.8]
    \draw[densely dashed, thick, double] (0,0) -- (0,2.15)  node[pos=.5, shape=coordinate](DOT){};
    \node at (0,-.25) {$\scs \lambda_k$};
    \filldraw  (DOT) circle (3pt);
    \node at (.5,1) {$\scs E_j^{(i)}$};
\end{tikzpicture}}
\end{equation}

For instance, when $\lambda=(\lambda_1,\lambda_2)$, $\nu=(\nu_1,\nu_2,\nu_3)$ and $\mathbf{i}=(\nu_3,\lambda_1,\lambda_2,\nu_1,\nu_2)\in \mathrm{Seq}(\lambda,\nu)$, we have two generating sets $E_{\lambda_1}$ and $E_{\lambda_2}$.
The elementary symmetric functions $E_{1,j}^{(i)}$ and $E_{2,j}^{(i)}$ are described diagrammatically as follows.

\begin{equation}
\nonumber
E_{1,j}^{(i)} e(\mathbf{i})=
\hackcenter{\begin{tikzpicture}[scale=0.8]
    \draw[thick] (0,0) -- (0,2.15);
    \node at (0,-.25) {$\scs \nu_3$};
    \draw[densely dashed, thick, double] (1,0) -- (1,2.15) node[pos=.5, shape=coordinate](DOT){};
    \node at (1,-.25) {$\scs \lambda_1$};
    \draw[densely dashed, thick, double] (2,0) -- (2,2.15);
    \node at (2,-.25) {$\scs \lambda_2$};
    \draw[thick] (3,0) -- (3,2.15);
    \node at (3,-.25) {$\scs \nu_1$};
    \draw[thick] (4,0) -- (4,2.15);
    \node at (4,-.25) {$\scs \nu_2$};
   \filldraw  (DOT) circle (3pt);
    \node at (1.5,1) {$\scs E_j^{(i)}$};
\end{tikzpicture}}\qquad
E_{2,j}^{(i)} e(\mathbf{i})=
\hackcenter{\begin{tikzpicture}[scale=0.8]
    \draw[thick] (0,0) -- (0,2.15);
    \node at (0,-.25) {$\scs \nu_3$};
    \draw[densely dashed, thick, double] (1,0) -- (1,2.15);
    \node at (1,-.25) {$\scs \lambda_1$};
    \draw[densely dashed, thick, double] (2,0) -- (2,2.15) node[pos=.5, shape=coordinate](DOT){};
    \node at (2,-.25) {$\scs \lambda_2$};
    \draw[thick] (3,0) -- (3,2.15);
    \node at (3,-.25) {$\scs \nu_1$};
    \draw[thick] (4,0) -- (4,2.15);
    \node at (4,-.25) {$\scs \nu_2$};
   \filldraw  (DOT) circle (3pt);
    \node at (2.5,1) {$\scs E_j^{(i)}$};
\end{tikzpicture}}
\end{equation}

\subsection{Thick dashed and solid crossings}
For $\mathbf{i}=(\mathbf{i}_1,\mathbf{i}_2)$, we define the element $\Psi e(\mathbf{i})$ by
\begin{equation}
\Psi e(\mathbf{i})=
\left\{\begin{array}{ll}
\psi_1 e(\mathbf{i}) &\text{if }\mathbf{i}_1,\mathbf{i}_2\in I,\\
\psi_{1}\psi_{2}\cdots \psi_{|\mathbf{i}_1|-1}\psi_{|\mathbf{i}_1|} e(\mathbf{i}) &\text{if }\mathbf{i}_1\in \Z_{\geq 0}[\overline{I}],\mathbf{i}_2\in I,\\
\psi_{|\mathbf{i}_2|}\psi_{|\mathbf{i}_2|-1}\cdots \psi_{2}\psi_{1}  e(\mathbf{i}) &\text{if }\mathbf{i}_1\in I,\mathbf{i}_2\in \Z_{\geq 0}[\overline{I}],\\
0&\text{if }\mathbf{i}_1,\mathbf{i}_2\in \Z_{\geq 0}[\overline{I}],
\end{array}\right.
\end{equation}
where $\psi_j$ is the generator of Khovanov-Lauda-Rouquier algebra $R(\overline{Q}_{\tilde{C}})$.

The diagrammatic description of this element $\Psi e(\mathbf{i})$ is defined as follows:\\
\noindent
$\bullet$ The case of $\mathbf{i}_1,\mathbf{i}_2\in I$:
\begin{equation}
\Psi e(\mathbf{i})=
\hackcenter{\begin{tikzpicture}[scale=0.6]
    \draw[thick] (0,0) .. controls (0,.75) and (1.5,.75) .. (1.5,1.5);
    \draw[thick] (1.5,0) .. controls (1.5,.75) and (0,.75) .. (0,1.5);
    \node at (0,-.25) {$\scs  \mathbf{i}_1$};
    \node at (1.5,-.25) {$\scs  \mathbf{i}_2$};
\end{tikzpicture}} .
\end{equation}

\noindent
$\bullet$ The case of $\mathbf{i}_1\in \Z_{\geq 0}[\overline{I}],\mathbf{i}_2\in I$:
\begin{equation}
\Psi e(\mathbf{i})=
\hackcenter{\begin{tikzpicture}[scale=0.6]
    \draw[densely dashed, thick, double] (0,0) .. controls (0,.75) and (1.5,.75) .. (1.5,1.5);
    \draw[thick] (1.5,0) .. controls (1.5,.75) and (0,.75) .. (0,1.5);
    \node at (0,-.25) {$\scs  \mathbf{i}_1$};
    \node at (1.5,-.25) {$\scs  \mathbf{i}_2$};
\end{tikzpicture}} =
\hackcenter{\begin{tikzpicture}[scale=0.8]
    \draw[densely dashed,thick] (0,0) -- (0,2.15) .. controls (0,2.65) and (.5,2.65) .. (.5,3.15);
    \draw[densely dashed,thick] (1.5,0) -- (1.5,2.15) .. controls (1.5,2.65) and (2,2.65) .. (2,3.15);
    \node at (0,-.25) {$\scs \overline{1}$};
    \node at (1.5,-.25) {$\scs \overline{1}$};
    \node at (0.75,1.85) {$\cdots$};
    \node at (0.75,0.25) {$\cdots$};
    \fill[white] (-0.1,0.5) rectangle (1.6,1.5);
    \draw[thick] (-0.1,0.5) rectangle (1.6,1.5);
    \node at (.75,1) {\Large $e_{\mathbf{i}_1^{(1)}}$};
    \draw[densely dashed,thick] (2,0) -- (2,2.15) .. controls (2,2.65) and (2.5,2.65) .. (2.5,3.15);
    \draw[densely dashed,thick] (3.5,0) -- (3.5,2.15) .. controls (3.5,2.65) and (4,2.65) .. (4,3.15);
    \node at (2,-.25) {$\scs \overline{2}$};
    \node at (3.5,-.25) {$\scs \overline{2}$};
    \node at (2.75,1.85) {$\cdots$};
    \node at (2.75,0.25) {$\cdots$};
    \fill[white] (1.9,0.5) rectangle (3.6,1.5);
    \draw[thick] (1.9,0.5) rectangle (3.6,1.5);
    \node at (2.75,1) {\Large $e_{\mathbf{i}_1^{(2)}}$};
    \draw[densely dashed,thick] (4.5,0) -- (4.5,2.15) .. controls (4.5,2.65) and (5,2.65) .. (5,3.15);
    \draw[densely dashed,thick] (6,0) -- (6,2.15) .. controls (6,2.65) and (6.5,2.65) .. (6.5,3.15);
    \node at (4.5,-.25) {$\scs \overline{N}$};
    \node at (6,-.25) {$\scs \overline{N}$};
    \node at (5.25,1.85) {$\cdots$};
    \node at (5.25,0.25) {$\cdots$};
    \fill[white] (4.4,0.5) rectangle (6.1,1.5);
    \draw[thick] (4.4,0.5) rectangle (6.1,1.5);
    \node at (5.25,1) {\Large $e_{\mathbf{i}_1^{(N)}}$};
    \node at (4,1) {$\cdots$};
     \draw[thick] (6.5,0) -- (6.5,2.15);
    \node at (6.5,-.25) {$\scs \mathbf{i}_2$};
    \draw[thick] (6.5,2.15) .. controls (6.5,3) and (0,2.45) .. (0,3.15);
\end{tikzpicture}}.
\end{equation}

\noindent
$\bullet$ The case of $\mathbf{i}_1\in I,\mathbf{i}_2\in \Z_{\geq 0}[\overline{I}]$:
\begin{equation}
\Psi e(\mathbf{i})=
\hackcenter{\begin{tikzpicture}[scale=0.6]
    \draw[thick] (0,0) .. controls (0,.75) and (1.5,.75) .. (1.5,1.5);
    \draw[densely dashed, thick, double] (1.5,0) .. controls (1.5,.75) and (0,.75) .. (0,1.5);
    \node at (0,-.25) {$\scs  \mathbf{i}_1$};
    \node at (1.5,-.25) {$\scs  \mathbf{i}_2$};
\end{tikzpicture}} =
\hackcenter{\begin{tikzpicture}[scale=0.8]
    \draw[densely dashed,thick] (0,0) -- (0,2.15) .. controls (0,2.65) and (-.5,2.65) .. (-.5,3.15);
    \draw[densely dashed,thick] (1.5,0) -- (1.5,2.15) .. controls (1.5,2.65) and (1,2.65) .. (1,3.15);
    \node at (0,-.25) {$\scs \overline{1}$};
    \node at (1.5,-.25) {$\scs \overline{1}$};
    \node at (0.75,1.85) {$\cdots$};
    \node at (0.75,0.25) {$\cdots$};
    \fill[white] (-0.1,0.5) rectangle (1.6,1.5);
    \draw[thick] (-0.1,0.5) rectangle (1.6,1.5);
    \node at (.75,1) {\Large $e_{\mathbf{i}_2^{(1)}}$};
    \draw[densely dashed,thick] (2,0) -- (2,2.15) .. controls (2,2.65) and (1.5,2.65) .. (1.5,3.15);
    \draw[densely dashed,thick] (3.5,0) -- (3.5,2.15) .. controls (3.5,2.65) and (3,2.65) .. (3,3.15);
    \node at (2,-.25) {$\scs \overline{2}$};
    \node at (3.5,-.25) {$\scs \overline{2}$};
    \node at (2.75,1.85) {$\cdots$};
    \node at (2.75,0.25) {$\cdots$};
    \fill[white] (1.9,0.5) rectangle (3.6,1.5);
    \draw[thick] (1.9,0.5) rectangle (3.6,1.5);
    \node at (2.75,1) {\Large $e_{\mathbf{i}_2^{(2)}}$};
    \draw[densely dashed,thick] (4.5,0) -- (4.5,2.15) .. controls (4.5,2.65) and (4,2.65) .. (4,3.15);
    \draw[densely dashed,thick] (6,0) -- (6,2.15) .. controls (6,2.65) and (5.5,2.65) .. (5.5,3.15);
    \node at (4.5,-.25) {$\scs \overline{N}$};
    \node at (6,-.25) {$\scs \overline{N}$};
    \node at (5.25,1.85) {$\cdots$};
    \node at (5.25,0.25) {$\cdots$};
    \fill[white] (4.4,0.5) rectangle (6.1,1.5);
    \draw[thick] (4.4,0.5) rectangle (6.1,1.5);
    \node at (5.25,1) {\Large $e_{\mathbf{i}_2^{(N)}}$};
    \node at (4,1) {$\cdots$};
     \draw[thick] (-.5,0) -- (-.5,2.15);
    \node at (-.5,-.25) {$\scs \mathbf{i}_1$};
    \draw[thick] (-.5,2.15) .. controls (-.5,3) and (6,2.45) .. (6,3.15);
\end{tikzpicture}}.
\end{equation}

\noindent
$\bullet$ The case of $\mathbf{i}_1,\mathbf{i}_2\in \Z_{\geq 0}[\overline{I}]$: We do not use the thick-thick dashed crossing since we have $\Psi e(\mathbf{i})=0$ for $\mathbf{i}\in (\Z_{\geq 0}[\overline{I}])^2$.

For an idempotent $e(\mathbf{i})\in \mathrm{Seq}(\lambda, \nu)$, we naturally extend the element $\Psi$ to $\Psi_j$ ($1\leq j\leq m+\ell-1$). 

For instance, when $\lambda=(\lambda_1,\lambda_2)$, $\nu=(\nu_1,\nu_2,\nu_3)$ and $\mathbf{i}=(\nu_3,\lambda_1,\lambda_2,\nu_1,\nu_2)\in \mathrm{Seq}(\lambda,\nu)$, the element $\Psi_j$ ($j=1,2,3,4$) is described diagrammatically as follows.

\begin{equation}
\nonumber
\Psi_1 e(\mathbf{i})=
\hackcenter{\begin{tikzpicture}[scale=0.8]
    \draw[thick] (0,0) .. controls (0,.75) and (1,1.4) .. (1,2.15);
    \node at (0,-.25) {$\scs \nu_3$};
    \draw[densely dashed, thick, double] (1,0)  .. controls (1,.75) and (0,1.4) .. (0,2.15);
    \node at (1,-.25) {$\scs \lambda_1$};
    \draw[densely dashed, thick, double] (2,0) -- (2,2.15);
    \node at (2,-.25) {$\scs \lambda_2$};
    \draw[thick] (3,0) -- (3,2.15);
    \node at (3,-.25) {$\scs \nu_1$};
    \draw[thick] (4,0) -- (4,2.15);
    \node at (4,-.25) {$\scs \nu_2$};
\end{tikzpicture}}\quad
\Psi_3 e(\mathbf{i})=
\hackcenter{\begin{tikzpicture}[scale=0.8]
    \draw[thick] (0,0) -- (0,2.15);
    \node at (0,-.25) {$\scs \nu_3$};
    \draw[densely dashed, thick, double] (1,0) -- (1,2.15);
    \node at (1,-.25) {$\scs \lambda_1$};
    \draw[densely dashed, thick, double] (2,0)  .. controls (2,.75) and (3,1.4) .. (3,2.15);
    \node at (2,-.25) {$\scs \lambda_2$};
    \draw[thick] (3,0)  .. controls (3,.75) and (2,1.4) .. (2,2.15);
    \node at (3,-.25) {$\scs \nu_1$};
    \draw[thick] (4,0) -- (4,2.15);
    \node at (4,-.25) {$\scs \nu_2$};
\end{tikzpicture}}\quad
\Psi_4 e(\mathbf{i})=
\hackcenter{\begin{tikzpicture}[scale=0.8]
    \draw[thick] (0,0) -- (0,2.15);
    \node at (0,-.25) {$\scs \nu_3$};
    \draw[densely dashed, thick, double] (1,0) -- (1,2.15);
    \node at (1,-.25) {$\scs \lambda_1$};
    \draw[densely dashed, thick, double] (2,0) -- (2,2.15);
    \node at (2,-.25) {$\scs \lambda_2$};
    \draw[thick] (3,0)  .. controls (3,.75) and (4,1.4) .. (4,2.15);
    \node at (3,-.25) {$\scs \nu_1$};
    \draw[thick] (4,0)  .. controls (4,.75) and (3,1.4) .. (3,2.15);
    \node at (4,-.25) {$\scs \nu_2$};
\end{tikzpicture}}
\end{equation}

Remark that we do not have the diagrammatic description for $\Psi_2e(\mathbf{i})$ since $\mathbf{i}_2,\mathbf{i}_3\in  \Z_{\geq 0}[\overline{I}]$.

\begin{remark}
The element $\Psi e(\mathbf{i})$ ($\mathbf{i}_1,\mathbf{i}_2\in \Z_{\geq 0}[\overline{I}]$) is not considered to be the following element of $R_{|\mathbf{i}_1|+|\mathbf{i}_1|}(\overline{Q}_{\tilde{C}})$:

\begin{equation}
\nonumber
\Psi e(\mathbf{i})\not=
\hackcenter{\begin{tikzpicture}[scale=0.8]
    \draw[densely dashed,thick] (0,0) -- (0,2.15) .. controls (0,3.5) and (4.5,2.65) .. (4.5,3.5);
    \draw[densely dashed,thick] (1.5,0) -- (1.5,2.15) .. controls (1.5,3) and (6,2.65) .. (6,3.5);
    \node at (0,-.25) {$\scs \overline{1}$};
    \node at (1.5,-.25) {$\scs \overline{1}$};
    \node at (0.75,1.85) {$\cdots$};
    \node at (0.75,0.25) {$\cdots$};
    \fill[white] (-0.1,0.5) rectangle (1.6,1.5);
    \draw[thick] (-0.1,0.5) rectangle (1.6,1.5);
    \node at (.75,1) {\Large $e_{\mathbf{i}_1^{(1)}}$};
    \draw[densely dashed,thick] (2.5,0) -- (2.5,2.15) .. controls (2.5,2.65) and (7,2.65) .. (7,3.5);
    \draw[densely dashed,thick] (4,0) -- (4,2.15) .. controls (4,2.65) and (8.5,2.65) .. (8.5,3.5);
    \node at (2.5,-.25) {$\scs \overline{N}$};
    \node at (4,-.25) {$\scs \overline{N}$};
    \node at (3.25,1.85) {$\cdots$};
    \node at (3.25,0.25) {$\cdots$};
    \fill[white] (2.4,0.5) rectangle (4.1,1.5);
    \draw[thick] (2.4,0.5) rectangle (4.1,1.5);
    \node at (3.25,1) {\Large $e_{\mathbf{i}_1^{(N)}}$};
    \node at (2,1) {$\cdots$};
    \draw[densely dashed,thick] (4.5,0) -- (4.5,2.15) .. controls (4.5,2.65) and (0,2.65) .. (0,3.5);
    \draw[densely dashed,thick] (6,0) -- (6,2.15) .. controls (6,2.65) and (1.5,2.65) .. (1.5,3.5);
    \node at (4.5,-.25) {$\scs \overline{1}$};
    \node at (6,-.25) {$\scs \overline{1}$};
    \node at (5.25,1.85) {$\cdots$};
    \node at (5.25,0.25) {$\cdots$};
    \fill[white] (4.4,0.5) rectangle (6.1,1.5);
    \draw[thick] (4.4,0.5) rectangle (6.1,1.5);
    \node at (5.25,1) {\Large $e_{\mathbf{i}_2^{(1)}}$};
    \draw[densely dashed,thick] (7,0) -- (7,2.15) .. controls (7,3) and (2.5,2.65) .. (2.5,3.5);
    \draw[densely dashed,thick] (8.5,0) -- (8.5,2.15) .. controls (8.5,3.5) and (4,2.65) .. (4,3.5);
    \node at (7,-.25) {$\scs \overline{N}$};
    \node at (8.5,-.25) {$\scs \overline{N}$};
    \node at (7.75,1.85) {$\cdots$};
    \node at (7.75,0.25) {$\cdots$};
    \fill[white] (6.9,0.5) rectangle (8.6,1.5);
    \draw[thick] (6.9,0.5) rectangle (8.6,1.5);
    \node at (7.75,1) {\Large $e_{\mathbf{i}_2^{(N)}}$};
    \node at (6.5,1) {$\cdots$};
    \node at (0.75,3.4) {$\cdots$};
    \node at (3.25,3.4) {$\cdots$};
    \node at (5.25,3.4) {$\cdots$};
    \node at (7.75,3.4) {$\cdots$};
\end{tikzpicture}}.
\end{equation}
\end{remark}

\subsection{Khovanov-Lauda-Rouquier subalgebra $S_{C}(\lambda,\nu)$}

\begin{definition}
Let $\lambda=(\lambda_1,...,\lambda_{m})$ be an element in $(\Z_{\geq 0}[\bar{I}])^m$, and let $\nu=(\nu_1,...,\nu_\ell)$ be an element in $I^{\ell}$.
The Khovanov-Lauda-Rouquier subalgebra $S_{C}(\lambda,\nu)$ of $R_{|\lambda|+\ell}(\overline{Q}_{\tilde{C}})$ is generated by
\begin{itemize}
\item $e(\mathbf{i})$ ($\mathbf{i}\in \mathrm{Seq}(\lambda,\nu)$),
\item $y_j$ ($1\leq j\leq m+\ell$),
\item The elementary symmetric functions in the generating sets $E_{\lambda_k}$ ($1\leq k\leq m$),
\item $\Psi_j$ ($1\leq j\leq m+\ell-1$).
\end{itemize}
\end{definition}

The Khovanov-Lauda-Rouquier subalgebra $S_{C}(\lambda,\nu)$ has the following properties.
\begin{proposition}

\begin{align}
\label{blackdot}
\hackcenter{\begin{tikzpicture}[scale=0.8]
    \draw[thick] (0,0) .. controls (0,.75) and (.75,.75) .. (.75,1.5)
        node[pos=.25, shape=coordinate](DOT){};
    \draw[densely dashed, thick, double] (.75,0) .. controls (.75,.75) and (0,.75) .. (0,1.5);
    \filldraw  (DOT) circle (2.5pt);
    \node at (0.8,-.2) {$\scs \lambda_k$};
    \node at (0,-.2) {$\scs i$};
\end{tikzpicture}}
\quad =\quad
\hackcenter{\begin{tikzpicture}[scale=0.8]
    \draw[thick] (0,0) .. controls (0,.75) and (.75,.75) .. (.75,1.5)
        node[pos=.75, shape=coordinate](DOT){};
    \draw[densely dashed, thick, double] (.75,0) .. controls (.75,.75) and (0,.75) .. (0,1.5);
    \filldraw  (DOT) circle (2.5pt);
    \node at (0.8,-.2) {$\scs \lambda_k$};
    \node at (0,-.2) {$\scs i$};
\end{tikzpicture}}
\qquad \quad
\hackcenter{\begin{tikzpicture}[scale=0.8]
    \draw[densely dashed, thick, double] (0,0) .. controls (0,.75) and (.75,.75) .. (.75,1.5);
    \draw[thick] (.75,0) .. controls (.75,.75) and (0,.75) .. (0,1.5)
        node[pos=.75, shape=coordinate](DOT){};
    \filldraw  (DOT) circle (2.75pt);
    \node at (0,-.2) {$\scs \lambda_k$};
    \node at (.8,-.2) {$\scs i$};
\end{tikzpicture}}
\quad=\quad
\hackcenter{\begin{tikzpicture}[scale=0.8]
    \draw[densely dashed, thick, double] (0,0) .. controls (0,.75) and (.75,.75) .. (.75,1.5);
    \draw[thick] (.75,0) .. controls (.75,.75) and (0,.75) .. (0,1.5)
        node[pos=.25, shape=coordinate](DOT){};
      \filldraw  (DOT) circle (2.75pt);
    \node at (0,-.2) {$\scs \lambda_k$};
    \node at (.8,-.2) {$\scs i$};
\end{tikzpicture}}
\end{align}

\begin{equation}
\label{thickdot}
\hackcenter{\begin{tikzpicture}[scale=0.8]
    \draw[densely dashed, thick, double] (0,0) .. controls (0,.75) and (.75,.75) .. (.75,1.5)
        node[pos=.25, shape=coordinate](DOT){};
    \draw[thick] (.75,0) .. controls (.75,.75) and (0,.75) .. (0,1.5);
    \filldraw  (DOT) circle (2.5pt);
    \node at (-.5,.5) {$\scs E_d^{(j)}$};
    \node at (0,-.2) {$\scs \lambda_k$};
    \node at (0.8,-.2) {$\scs i$};
\end{tikzpicture}}
\quad =\quad
\hackcenter{\begin{tikzpicture}[scale=0.8]
    \draw[densely dashed, thick, double] (0,0) .. controls (0,.75) and (.75,.75) .. (.75,1.5)
        node[pos=.75, shape=coordinate](DOT){};
    \draw[thick] (.75,0) .. controls (.75,.75) and (0,.75) .. (0,1.5);
    \filldraw  (DOT) circle (2.5pt);
    \node at (1.25,1) {$\scs E_d^{(j)}$};
    \node at (0,-.2) {$\scs \lambda_k$};
    \node at (0.8,-.2) {$\scs i$};
\end{tikzpicture}}
\qquad \quad
\hackcenter{\begin{tikzpicture}[scale=0.8]
    \draw[thick] (0,0) .. controls (0,.75) and (.75,.75) .. (.75,1.5);
    \draw[densely dashed, thick, double] (.75,0) .. controls (.75,.75) and (0,.75) .. (0,1.5)
        node[pos=.75, shape=coordinate](DOT){};
    \filldraw  (DOT) circle (2.75pt);
    \node at (-.5,1) {$\scs E_d^{(j)}$};
    \node at (0.8,-.2) {$\scs \lambda_k$};
    \node at (0,-.2) {$\scs i$};
\end{tikzpicture}}
\quad=\quad
\hackcenter{\begin{tikzpicture}[scale=0.8]
    \draw[thick] (0,0) .. controls (0,.75) and (.75,.75) .. (.75,1.5);
    \draw[densely dashed, thick, double] (.75,0) .. controls (.75,.75) and (0,.75) .. (0,1.5)
        node[pos=.25, shape=coordinate](DOT){};
      \filldraw  (DOT) circle (2.75pt);
    \node at (1.25,.5) {$\scs E_d^{(j)}$};
    \node at (0.8,-.2) {$\scs \lambda_k$};
    \node at (0,-.2) {$\scs i$};
\end{tikzpicture}}
\end{equation}

\begin{equation}\label{thinthick2-rel}
\hackcenter{\begin{tikzpicture}[scale=0.8]
    \draw[thick] (0,0) .. controls ++(0,.5) and ++(0,-.5) .. (.75,1);
    \draw[densely dashed, thick, double] (.75,0) .. controls ++(0,.5) and ++(0,-.5) .. (0,1);
    \draw[densely dashed, thick, double] (0,1 ) .. controls ++(0,.5) and ++(0,-.5) .. (.75,2);
    \draw[thick] (.75,1) .. controls ++(0,.5) and ++(0,-.5) .. (0,2);
    \node at (0,-.25) {$\;$};
    \node at (0,2.25) {$\;$};
    \node at (.8,-.2) {$\scs \lambda_k$};
    \node at (0,-.2) {$\scs i$};
\end{tikzpicture}}
 \;\; = \;\;
\sum_{d_1+d_2=\lambda_k^{(i)}}(-1)^{d_2}
\hackcenter{\begin{tikzpicture}[scale=0.8]
    \draw[thick] (0,0) -- (0,1.5)  node[pos=.55, shape=coordinate](DOT){};
    \filldraw  (DOT) circle (2.5pt);
    \draw[densely dashed, thick, double] (.75,0) -- (.75,1.5) node[pos=.55, shape=coordinate](DOT2){};
    \filldraw  (DOT2) circle (2.75pt);
    \node at (-.3,.75) {$\scs d_1$};
    \node at (1.25,.75) {$\scs {E}_{d_2}^{(i)}$};
    \node at (.8,-.2) {$\scs \lambda_k$};
    \node at (0,-.2) {$\scs i$};
\end{tikzpicture}}
%
\qquad \quad
\hackcenter{\begin{tikzpicture}[scale=0.8]
    \draw[densely dashed, thick, double] (0,0) .. controls ++(0,.5) and ++(0,-.5) .. (.75,1);
    \draw[thick] (.75,0) .. controls ++(0,.5) and ++(0,-.5) .. (0,1);
    \draw[thick] (0,1 ) .. controls ++(0,.5) and ++(0,-.5) .. (.75,2);
    \draw[densely dashed, thick, double] (.75,1) .. controls ++(0,.5) and ++(0,-.5) .. (0,2);
    \node at (0,-.25) {$\;$};
    \node at (0,2.25) {$\;$};
    \node at (0,-.2) {$\scs \lambda_k$};
    \node at (.8,-.2) {$\scs i$};
\end{tikzpicture}}
 \;\; = \;\;
\sum_{d_1+d_2=\lambda_k^{(i)}}(-1)^{d_1}
\hackcenter{\begin{tikzpicture}[scale=0.8]
    \draw[thick] (.75,0) -- (.75,1.5)  node[pos=.55, shape=coordinate](DOT){};
    \filldraw  (DOT) circle (2.5pt);
    \draw[densely dashed, thick, double] (0,0) -- (0,1.5) node[pos=.55, shape=coordinate](DOT2){};
    \filldraw  (DOT2) circle (2.75pt);
    \node at (1.15,.75) {$\scs d_2$};
    \node at (-.5,.75) {$\scs {E}_{d_1}^{(i)}$};
    \node at (0,-.2) {$\scs \lambda_k$};
    \node at (.8,-.2) {$\scs i$};
\end{tikzpicture}}
\end{equation}

\begin{align}\label{r3}
\hackcenter{\begin{tikzpicture}[scale=0.8]
    \draw[densely dashed, thick, double] (0,0) .. controls ++(0,1) and ++(0,-1) .. (1.5,2);
    \draw[thick] (.75,0) .. controls ++(0,.5) and ++(0,-.5) .. (0,1.0);
    \draw[thick] (0,1.0) .. controls ++(0,.5) and ++(0,-.5) .. (0.75,2);
    \draw[thick] (1.5,0) .. controls ++(0,1) and ++(0,-1) .. (0,2);
    \node at (0,-.2) {$\scs \lambda_k$};
    \node at (0.75,-.2) {$\scs i$};
    \node at (1.5,-.2) {$\scs j$};
\end{tikzpicture}}
\;\; = \;\;
\hackcenter{\begin{tikzpicture}[scale=0.8]
    \draw[densely dashed, thick, double] (0,0) .. controls ++(0,1) and ++(0,-1) .. (1.5,2);
    \draw[thick] (.75,0) .. controls ++(0,.5) and ++(0,-.5) .. (1.5,1.0);
    \draw[thick] (1.5,1.0) .. controls ++(0,.5) and ++(0,-.5) .. (0.75,2.0);
    \draw[thick] (1.5,0) .. controls ++(0,1) and ++(0,-1) .. (0,2.0);
    \node at (0,-.2) {$\scs \lambda_k$};
    \node at (0.75,-.2) {$\scs i$};
    \node at (1.5,-.2) {$\scs j$};
\end{tikzpicture}}
\qquad \qquad
\hackcenter{\begin{tikzpicture}[scale=0.8]
    \draw[thick] (0,0) .. controls ++(0,1) and ++(0,-1) .. (1.5,2);
    \draw[thick] (.75,0) .. controls ++(0,.5) and ++(0,-.5) .. (0,1.0);
    \draw[thick] (0,1.0) .. controls ++(0,.5) and ++(0,-.5) .. (0.75,2);
    \draw[densely dashed, thick, double] (1.5,0) .. controls ++(0,1) and ++(0,-1) .. (0,2);
    \node at (1.5,-.2) {$\scs \lambda_k$};
    \node at (0.7,-.2) {$\scs i$};
    \node at (0,-.2) {$\scs j$};
\end{tikzpicture}}
\;\; = \;\;
\hackcenter{\begin{tikzpicture}[scale=0.8]
    \draw[thick] (0,0) .. controls ++(0,1) and ++(0,-1) .. (1.5,2);
    \draw[thick] (.75,0) .. controls ++(0,.5) and ++(0,-.5) .. (1.5,1.0);
    \draw[thick] (1.5,1.0) .. controls ++(0,.5) and ++(0,-.5) .. (0.75,2.0);
    \draw[densely dashed, thick, double] (1.5,0) .. controls ++(0,1) and ++(0,-1) .. (0,2.0);
    \node at (1.2,-.2) {$\scs \lambda_k$};
    \node at (0.6,-.2) {$\scs i$};
    \node at (0,-.2) {$\scs j$};
\end{tikzpicture}}
\end{align}

\begin{equation}\label{r3-2}
\hackcenter{\begin{tikzpicture}[scale=0.8]
    \draw[thick] (0,0) .. controls ++(0,1) and ++(0,-1) .. (1.5,2);
    \draw[densely dashed, thick, double] (.75,0) .. controls ++(0,.5) and ++(0,-.5) .. (0,1.0);
    \draw[densely dashed, thick, double] (0,1.0) .. controls ++(0,.5) and ++(0,-.5) .. (0.75,2);
    \draw[thick] (1.5,0) .. controls ++(0,1) and ++(0,-1) .. (0,2);
    \node at (0.75,-.2) {$\scs \lambda_k$};
    \node at (0,-.2) {$\scs i$};
    \node at (1.5,-.2) {$\scs j$};
\end{tikzpicture}}
\;\; - \;\;
\hackcenter{\begin{tikzpicture}[scale=0.8]
    \draw[thick] (0,0) .. controls ++(0,1) and ++(0,-1) .. (1.5,2);
    \draw[densely dashed, thick, double] (.75,0) .. controls ++(0,.5) and ++(0,-.5) .. (1.5,1.0);
    \draw[densely dashed, thick, double] (1.5,1.0) .. controls ++(0,.5) and ++(0,-.5) .. (0.75,2.0);
    \draw[thick] (1.5,0) .. controls ++(0,1) and ++(0,-1) .. (0,2.0);
    \node at (0.75,-.2) {$\scs \lambda_k$};
    \node at (0,-.2) {$\scs i$};
    \node at (1.5,-.2) {$\scs j$};
\end{tikzpicture}}
\;\; = \;\; \delta_{i,j}\sum_{a+b+c=\lambda_k^{(i)}-1}(-1)^c
\hackcenter{\begin{tikzpicture}[scale=0.8]
    \draw[thick] (0,0) -- (0,2)  node[pos=.5, shape=coordinate](DOT){};
    \draw[densely dashed, thick, double] (.75,0) --  (.75,2) node[pos=.75, shape=coordinate](DOT2){};
    \draw[thick] (1.5,0) -- (1.5,2)  node[pos=.5, shape=coordinate](DOT1){};
    \filldraw  (DOT) circle (2.5pt);
    \filldraw  (DOT1) circle (2.5pt);
    \filldraw  (DOT2) circle (2.5pt);
 \node at (-.4,1) {$\scs a$};
    \node at (1.9,1) {$\scs b$};
        \node at (1.2,1.5) {$\scs E_c^{(i)}$};
    \node at (0.75,-.2) {$\scs \lambda_k$};
    \node at (0,-.2) {$\scs i$};
    \node at (1.5,-.2) {$\scs i$};
\end{tikzpicture}}
\end{equation}
\end{proposition}

\begin{proof}
We show the first equation of \eqref{thinthick2-rel}.
In the case of $\lambda_k^{(a)}=0$ if $a\not=i$, the left-hand side of equation is
$$
\hackcenter{\begin{tikzpicture}[scale=0.8]
    \draw[thick] (0,0) .. controls ++(0,.5) and ++(0,-.5) .. (.75,1);
    \draw[densely dashed, thick, double] (.75,0) .. controls ++(0,.5) and ++(0,-.5) .. (0,1);
    \draw[densely dashed, thick, double] (0,1 ) .. controls ++(0,.5) and ++(0,-.5) .. (.75,2);
    \draw[thick] (.75,1) .. controls ++(0,.5) and ++(0,-.5) .. (0,2);
    \node at (0,-.25) {$\;$};
    \node at (0,2.25) {$\;$};
    \node at (.8,-.2) {$\scs \lambda_k$};
    \node at (0,-.2) {$\scs i$};
\end{tikzpicture}}=
\hackcenter{\begin{tikzpicture}[scale=0.8]
    \draw[densely dashed,thick] (0,0) -- (0,1.25) .. controls (0,1.5) and (-.5,1.5) .. (-.5,1.75);
    \draw[densely dashed,thick] (1.5,0) -- (1.5,1.25) .. controls (1.5,1.5) and (1,1.5) .. (1,1.75);
    \node at (0,-.25) {$\scs \overline{i}$};
    \node at (1.5,-.25) {$\scs \overline{i}$};
    \fill[white] (-0.1,0.25) rectangle (1.6,1.25);
    \draw[thick] (-0.1,0.25) rectangle (1.6,1.25);
    \node at (.75,.75) {\Large $e_{\lambda_k^{(i)}}$};
     \draw[thick] (-.5,0) -- (-.5,1.25) .. controls (-.5,1.5) and (1.5,1.5) .. (1.5,1.75);
    \node at (-.5,-.25) {$\scs i$};
    \draw[densely dashed,thick] (-.5,1.75) -- (-.5,3) .. controls (-.5,3.25) and (0,3.25) .. (0,3.5);
    \draw[densely dashed,thick] (1,1.75) -- (1,3) .. controls (1,3.25) and (1.5,3.25) .. (1.5,3.5);
    \fill[white] (-0.6,2) rectangle (1.1,3);
    \draw[thick] (-0.6,2) rectangle (1.1,3);
    \node at (.25,2.5) {\Large $e_{\lambda_k^{(i)}}$};
     \draw[thick] (1.5,1.75) -- (1.5,3) .. controls (1.5,3.25) and (-0.5,3.25) .. (-0.5,3.5);
\end{tikzpicture}}
=
\hackcenter{\begin{tikzpicture}[scale=0.8]
    \draw[densely dashed,thick] (0,0) -- (0,1.25) .. controls (0,1.5) and (-.5,1.5) .. (-.5,1.75);
    \draw[densely dashed,thick] (1.5,0) -- (1.5,1.25) .. controls (1.5,1.5) and (1,1.5) .. (1,1.75);
    \node at (0,-.25) {$\scs \overline{i}$};
    \node at (1.5,-.25) {$\scs \overline{i}$};
    \fill[white] (-0.1,0.25) rectangle (1.6,1.25);
    \draw[thick] (-0.1,0.25) rectangle (1.6,1.25);
    \node at (.75,.75) {\Large $e_{\lambda_k^{(i)}}$};
     \draw[thick] (-.5,0) -- (-.5,1.25) .. controls (-.5,1.5) and (1.5,1.5) .. (1.5,1.75);
    \node at (-.5,-.25) {$\scs i$};
    \draw[densely dashed,thick] (-.5,1.75) -- (-.5,3) .. controls (-.5,3.25) and (0,3.25) .. (0,3.5);
    \draw[densely dashed,thick] (1,1.75) -- (1,3) .. controls (1,3.25) and (1.5,3.25) .. (1.5,3.5);
    \node at (.25,2.5) {\Large $\cdots$};
     \draw[thick] (1.5,1.75) -- (1.5,3) .. controls (1.5,3.25) and (-0.5,3.25) .. (-0.5,3.5);
\end{tikzpicture}}
=
\sum_{d_1+d_2=\lambda_k^{(i)}}
(-1)^{d_2}
\hackcenter{\begin{tikzpicture}[scale=0.8]
    \draw[densely dashed,thick] (0,0) -- (0,3);
    \draw[densely dashed,thick] (1.5,0) -- (1.5,3);
    \node at (0,-.25) {$\scs \overline{i}$};
    \node at (1.5,-.25) {$\scs \overline{i}$};
    \fill[white] (-0.1,0.25) rectangle (1.6,1.25);
    \draw[thick] (-0.1,0.25) rectangle (1.6,1.25);
    \node at (.75,.75) {\Large $e_{\lambda_k^{(i)}}$};
     \draw[thick] (-.5,0) -- (-0.5,3) node[pos=.75, shape=coordinate](DOT){};
    \node at (-.5,-.25) {$\scs i$};
    \fill[white] (-0.1,1.75) rectangle (1.6,2.75);
    \draw[thick] (-0.1,1.75) rectangle (1.6,2.75);
    \node at (.75,2.25) { $E_{d_2}^{(i)}$};
    \filldraw  (DOT) circle (2.5pt);
    \node at (-.9,2.25) {$\scs d_1$};
\end{tikzpicture}}
,
$$
where the box of $E_{d_2}^{(i)}$ in the last term is represented by the $d_2$-th elementary symmetric function of the polynomial $\Bbbk[x_2,...,x_{\lambda_k^{(i)}+1}]$. 
We have the second equality using Equation \eqref{R3-gene} and the third equality using Equation \eqref{R2-thin-thick}.
This is exactly the same as the right-hand side of the first equation of \eqref{thinthick2-rel}.

For the general $\lambda_k$, the crossings composed of the solid strand labeled by $i$ and the dashed strand labeled by $\overline{j}$ ($j\not=i$) do not induce any dots since we have Equation \eqref{R2-thin-thick}.
Therefore, we have the first equation of \eqref{thinthick2-rel}.
The second equation of \eqref{thinthick2-rel} is proven similarly.

We show the equation of \eqref{r3-2}.
In the case of $i=j$ and $\lambda_k^{(a)}=0$ if $a\not=i$, the first term in the left-hand side of equation is
\begin{eqnarray*}
\hackcenter{\begin{tikzpicture}[scale=0.8]
    \draw[thick] (0,0) .. controls ++(0,1) and ++(0,-1) .. (1.5,2);
    \draw[densely dashed, thick, double] (.75,0) .. controls ++(0,.5) and ++(0,-.5) .. (0,1.0);
    \draw[densely dashed, thick, double] (0,1.0) .. controls ++(0,.5) and ++(0,-.5) .. (0.75,2);
    \draw[thick] (1.5,0) .. controls ++(0,1) and ++(0,-1) .. (0,2);
    \node at (0.75,-.2) {$\scs \lambda_k$};
    \node at (0,-.2) {$\scs i$};
    \node at (1.5,-.2) {$\scs j$};
\end{tikzpicture}}
&=&
\hackcenter{\begin{tikzpicture}[scale=0.8]
    \draw[densely dashed,thick] (0,0) -- (0,1.25) .. controls (0,1.5) and (-.5,1.5) .. (-.5,1.75);
    \draw[densely dashed,thick] (1.5,0) -- (1.5,1.25) .. controls (1.5,1.5) and (1,1.5) .. (1,1.75);
    \draw[densely dashed,thick] (1.25,0) -- (1.25,1.25) .. controls (1.25,1.5) and (0.75,1.5) .. (0.75,1.75);
    \node at (0,-.25) {$\scs \overline{i}$};
    \node at (1.5,-.25) {$\scs \overline{i}$};
    \node at (1.25,-.25) {$\scs \overline{i}$};
    \fill[white] (-0.1,0.25) rectangle (1.6,1.25);
    \draw[thick] (-0.1,0.25) rectangle (1.6,1.25);
    \node at (.75,.75) {\Large $e_{\lambda_k^{(i)}}$};
     \draw[thick] (-.5,0) -- (-.5,1.25) .. controls (-.5,1.5) and (2,1.5) .. (2,2.25) -- (2,2.5);
    \node at (-.5,-.25) {$\scs i$};
    \draw[densely dashed,thick] (-.5,1.75) -- (-.5,2) .. controls (-.5,2.25) and (0,2.25) .. (0,2.5);
    \draw[densely dashed,thick] (1,1.75) -- (1,2) .. controls (1,2.25) and (1.5,2.25) .. (1.5,2.5);
    \draw[densely dashed,thick] (.75,1.75) -- (.75,2) .. controls (.75,2.25) and (1.25,2.25) .. (1.25,2.5);
    \node at (.75,2.5) {\Large $\cdots$};
     \draw[thick] (2,0) -- (2,1.5) .. controls (2,2.25) and (-0.5,2.25) .. (-0.5,2.5);
    \node at (2,-.25) {$\scs i$};
\end{tikzpicture}}
=
\hackcenter{\begin{tikzpicture}[scale=0.8]
    \draw[densely dashed,thick] (0,0) -- (0,1.25) .. controls (0,1.5) and (-.5,1.5) .. (-.5,1.75);
    \draw[densely dashed,thick] (1.5,0) -- (1.5,2.5);
    \draw[densely dashed,thick] (1.25,0) -- (1.25,1.25) .. controls (1.25,1.5) and (0.75,1.5) .. (0.75,1.75);
    \node at (0,-.25) {$\scs \overline{i}$};
    \node at (1.5,-.25) {$\scs \overline{i}$};
    \node at (1.25,-.25) {$\scs \overline{i}$};
    \fill[white] (-0.1,0.25) rectangle (1.6,1.25);
    \draw[thick] (-0.1,0.25) rectangle (1.6,1.25);
    \node at (.75,.75) {\Large $e_{\lambda_k^{(i)}}$};
     \draw[thick] (-.5,0) -- (-.5,1.25) .. controls (-.5,1.5) and (1.25,1.5) .. (1.25,1.75) -- (1.25,2) .. controls (1.25,2.25) and (-.5,2.25) .. (-.5,2.5);
    \node at (-.5,-.25) {$\scs i$};
    \draw[densely dashed,thick] (-.5,1.75) -- (-.5,2) .. controls (-.5,2.25) and (0,2.25) .. (0,2.5);
    \draw[densely dashed,thick] (.75,1.75) -- (.75,2) .. controls (.75,2.25) and (1.25,2.25) .. (1.25,2.5);
    \node at (.75,2.5) {\Large $\cdots$};
     \draw[thick] (2,0) -- (2,2.5);
    \node at (2,-.25) {$\scs i$};
\end{tikzpicture}}
+
\hackcenter{\begin{tikzpicture}[scale=0.8]
    \draw[densely dashed,thick] (0,0) -- (0,1.25) .. controls (0,1.5) and (-.5,1.5) .. (-.5,1.75);
    \draw[densely dashed,thick] (1.5,0) -- (1.5,1.25) .. controls (1.5,1.5) and (1.75,1.5) .. (1.75,1.75);
    \draw[densely dashed,thick] (1.25,0) -- (1.25,1.25) .. controls (1.25,1.5) and (0.75,1.5) .. (0.75,1.75);
    \node at (0,-.25) {$\scs \overline{i}$};
    \node at (1.5,-.25) {$\scs \overline{i}$};
    \node at (1.25,-.25) {$\scs \overline{i}$};
    \fill[white] (-0.1,0.25) rectangle (1.6,1.25);
    \draw[thick] (-0.1,0.25) rectangle (1.6,1.25);
    \node at (.75,.75) {\Large $e_{\lambda_k^{(i)}}$};
     \draw[thick] (-.5,0) -- (-.5,1.25) .. controls (-.5,1.5) and (2,2) .. (2,2.25) -- (2,2.5);
    \node at (-.5,-.25) {$\scs i$};
    \draw[densely dashed,thick] (-.5,1.75) -- (-.5,2) .. controls (-.5,2.25) and (0,2.25) .. (0,2.5);
    \draw[densely dashed,thick] (1.75,1.75) -- (1.75,2) .. controls (1.75,2.25) and (1.5,2.25) .. (1.5,2.5);
    \draw[densely dashed,thick] (.75,1.75) -- (.75,2) .. controls (.75,2.25) and (1.25,2.25) .. (1.25,2.5);
    \node at (.75,2.5) {\Large $\cdots$};
     \draw[thick] (2,0) -- (2,1.5) .. controls (2,1.75) and (-0.5,2.25) .. (-0.5,2.5);
    \node at (2,-.25) {$\scs i$};
\end{tikzpicture}}\\
&=&
\sum_{d_1+d_2=\lambda_k^{(i)}-1}(-1)^{d_2}
\hackcenter{\begin{tikzpicture}[scale=0.8]
    \draw[densely dashed,thick] (0,0) -- (0,2.5);
    \draw[densely dashed,thick] (1.5,0) -- (1.5,2.5);
    \draw[densely dashed,thick] (1.25,0) -- (1.25,2.5);
    \node at (0,-.25) {$\scs \overline{i}$};
    \node at (1.5,-.25) {$\scs \overline{i}$};
    \node at (1.25,-.25) {$\scs \overline{i}$};
    \fill[white] (-0.1,0.25) rectangle (1.6,1.25);
    \draw[thick] (-0.1,0.25) rectangle (1.6,1.25);
    \node at (.75,.75) {\Large $e_{\lambda_k^{(i)}}$};
    \fill[white] (-0.1,1.375) rectangle (1.6,2.375);
    \draw[thick] (-0.1,1.375) rectangle (1.6,2.375);
    \node at (.75,1.875) {\Large $E_{d_2}^{(i)}$};
     \draw[thick] (-.5,0) -- (-.5,2.5) node[pos=.75, shape=coordinate](DOT){};
    \node at (-.5,-.25) {$\scs i$};
    \node at (.75,2.5) {\Large $\cdots$};
     \draw[thick] (2,0) -- (2,2.5);
    \node at (2,-.25) {$\scs i$};
    \node at (-.9,2) {$\scs d_1$};
    \filldraw  (DOT) circle (2.5pt);
\end{tikzpicture}}
+
\hackcenter{\begin{tikzpicture}[scale=0.8]
    \draw[densely dashed,thick] (0,0) -- (0,1.25) .. controls (0,1.5) and (-.5,1.5) .. (-.5,1.75);
    \draw[densely dashed,thick] (1.5,0) -- (1.5,1.25) .. controls (1.5,1.5) and (2,1.5) .. (2,1.75);
    \draw[densely dashed,thick] (1.25,0) -- (1.25,2.5);
    \draw[densely dashed,thick] (1,0) -- (1,1.25) .. controls (1,1.5) and (0.5,1.5) .. (0.5,1.75);
    \node at (0,-.25) {$\scs \overline{i}$};
    \node at (1.5,-.25) {$\scs \overline{i}$};
    \node at (1.25,-.25) {$\scs \overline{i}$};
    \node at (1,-.25) {$\scs \overline{i}$};
    \fill[white] (-0.1,0.25) rectangle (1.6,1.25);
    \draw[thick] (-0.1,0.25) rectangle (1.6,1.25);
    \node at (.75,.75) {\Large $e_{\lambda_k^{(i)}}$};
     \draw[thick] (-.5,0) -- (-.5,1.25) .. controls (-.5,1.5) and (1,1.5) .. (1,1.75) -- (1,2) .. controls (1,2.25) and (-.5,2.25) .. (-.5,2.5);
    \node at (-.5,-.25) {$\scs i$};
    \draw[densely dashed,thick] (-.5,1.75) -- (-.5,2) .. controls (-.5,2.25) and (0,2.25) .. (0,2.5);
    \draw[densely dashed,thick] (.5,1.75) -- (.5,2) .. controls (.5,2.25) and (1,2.25) .. (1,2.5);
    \draw[densely dashed,thick] (2,1.75) -- (2,2) .. controls (2,2.25) and (1.5,2.25) .. (1.5,2.5);
    \node at (.5,2.5) {\Large $\cdots$};
     \draw[thick] (2,0) -- (2,1.25) .. controls (2,1.5) and (1.5,1.5) .. (1.5,1.75) -- (1.5,2) .. controls (1.5,2.25) and (2,2.25) .. (2,2.5);
    \node at (2,-.25) {$\scs i$};
\end{tikzpicture}}
+
\hackcenter{\begin{tikzpicture}[scale=0.8]
    \draw[densely dashed,thick] (0,0) -- (0,1.25) .. controls (0,1.5) and (-.5,1.5) .. (-.5,1.75);
    \draw[densely dashed,thick] (1.5,0) -- (1.5,1.25) .. controls (1.5,1.5) and (1.75,1.5) .. (1.75,1.75);
    \draw[densely dashed,thick] (1.25,0) -- (1.25,1.25) .. controls (1.25,1.5) and (1.5,1.5) .. (1.5,1.75);
    \draw[densely dashed,thick] (1,0) -- (1,1.25) .. controls (1,1.5) and (0.5,1.5) .. (0.5,1.75);
    \node at (0,-.25) {$\scs \overline{i}$};
    \node at (1.5,-.25) {$\scs \overline{i}$};
    \node at (1.25,-.25) {$\scs \overline{i}$};
    \node at (1,-.25) {$\scs \overline{i}$};
    \fill[white] (-0.1,0.25) rectangle (1.6,1.25);
    \draw[thick] (-0.1,0.25) rectangle (1.6,1.25);
    \node at (.75,.75) {\Large $e_{\lambda_k^{(i)}}$};
     \draw[thick] (-.5,0) -- (-.5,1.25) .. controls (-.5,1.5) and (2,2) .. (2,2.25) -- (2,2.5);
    \node at (-.5,-.25) {$\scs i$};
    \draw[densely dashed
    ,thick] (-.5,1.75) -- (-.5,2) .. controls (-.5,2.25) and (0,2.25) .. (0,2.5);
    \draw[densely dashed,thick] (1.75,1.75) -- (1.75,2) .. controls (1.75,2.25) and (1.5,2.25) .. (1.5,2.5);
    \draw[densely dashed,thick] (1.5,1.75) -- (1.5,2) .. controls (1.5,2.25) and (1.25,2.25) .. (1.25,2.5);
    \draw[densely dashed,thick] (.5,1.75) -- (.5,2) .. controls (.5,2.25) and (1,2.25) .. (1,2.5);
    \node at (.5,2.5) {\Large $\cdots$};
     \draw[thick] (2,0) -- (2,1.5) .. controls (2,1.75) and (-0.5,2.25) .. (-0.5,2.5);
    \node at (2,-.25) {$\scs i$};
\end{tikzpicture}}\\
&=&
\sum_{d_1+d_2+d_3=\lambda_k^{(i)}-1}(-1)^{d_2}
\hackcenter{\begin{tikzpicture}[scale=0.8]
    \draw[densely dashed,thick] (0,0) -- (0,2.5);
    \draw[densely dashed,thick] (1.5,0) -- (1.5,2.5);
    \draw[densely dashed,thick] (1.25,0) -- (1.25,2.5);
    \node at (0,-.25) {$\scs \overline{i}$};
    \node at (1.5,-.25) {$\scs \overline{i}$};
    \node at (1.25,-.25) {$\scs \overline{i}$};
    \fill[white] (-0.1,0.25) rectangle (1.6,1.25);
    \draw[thick] (-0.1,0.25) rectangle (1.6,1.25);
    \node at (.75,.75) {\Large $e_{\lambda_k^{(i)}}$};
    \fill[white] (-0.1,1.375) rectangle (1.6,2.375);
    \draw[thick] (-0.1,1.375) rectangle (1.6,2.375);
    \node at (.75,1.875) {\Large $E_{d_2}^{(i)}$};
     \draw[thick] (-.5,0) -- (-.5,2.5) node[pos=.75, shape=coordinate](DOT){};
    \node at (-.5,-.25) {$\scs i$};
    \node at (.75,2.5) {\Large $\cdots$};
     \draw[thick] (2,0) -- (2,2.5) node[pos=.75, shape=coordinate](DOT1){};
    \node at (2,-.25) {$\scs i$};
    \node at (-.9,2) {$\scs d_1$};
    \node at (2.4,2) {$\scs d_3$};
    \filldraw  (DOT) circle (2.5pt);
    \filldraw  (DOT1) circle (2.5pt);
\end{tikzpicture}}
+
\hackcenter{\begin{tikzpicture}[scale=0.8]
    \draw[densely dashed,thick] (0,0) -- (0,1.25) .. controls (0,1.5) and (.5,1.5) .. (.5,1.75);
    \draw[densely dashed,thick] (1.5,0) -- (1.5,1.25) .. controls (1.5,1.5) and (2,1.5) .. (2,1.75);
    \draw[densely dashed,thick] (1.25,0) -- (1.25,1.25) .. controls (1.25,1.5) and (1.75,1.5) .. (1.75,1.75);
    \node at (0,-.25) {$\scs \overline{i}$};
    \node at (1.5,-.25) {$\scs \overline{i}$};
    \node at (1.25,-.25) {$\scs \overline{i}$};
    \fill[white] (-0.1,0.25) rectangle (1.6,1.25);
    \draw[thick] (-0.1,0.25) rectangle (1.6,1.25);
    \node at (.75,.75) {\Large $e_{\lambda_k^{(i)}}$};
     \draw[thick] (-.5,0) -- (-.5,1.5) .. controls (-.5,2.25) and (2,2.25) .. (2,2.5);
    \node at (-.5,-.25) {$\scs i$};
    \draw[densely dashed,thick] (.5,1.75) -- (.5,2) .. controls (.5,2.25) and (0,2.25) .. (0,2.5);
    \draw[densely dashed,thick] (2,1.75) -- (2,2) .. controls (2,2.25) and (1.5,2.25) .. (1.5,2.5);
    \draw[densely dashed,thick] (1.75,1.75) -- (1.75,2) .. controls (1.75,2.25) and (1.25,2.25) .. (1.25,2.5);
    \node at (.75,2.5) {\Large $\cdots$};
     \draw[thick] (2,0) -- (2,1.25) .. controls (2,1.5) and (-0.5,1.5) .. (-0.5,2.25) -- (-0.5,2.5);
    \node at (2,-.25) {$\scs i$};
\end{tikzpicture}}
\end{eqnarray*}
The first term is equal to the right-hand side of Equation \eqref{r3-2} and the second term is equal to the second term in the left-hand side of Equation \eqref{r3-2}.
For the general $\lambda_k$, Equation \eqref{r3-2} is proven similarly.

Equations \eqref{blackdot}, \eqref{thickdot} and \eqref{r3} are proven using Equations \eqref{dot-slide1}, \eqref{dot-slide2} and \eqref{RIII}.
\end{proof}

The Khovanov-Lauda-Rouquier subalgebra $S_{C}(\lambda,\nu)$ of $R_n(\overline{Q}_{\tilde{C}})$ is a generalization of the redotted Webster algebra of type $A_1$ defined in \cite{KLSY}.
\begin{theorem}
Let $I=\{i\}$ be Cartan datum of the type of $A_1$. 
For $\lambda_k=(\lambda_1\cdot i,...,\lambda_m\cdot i)\in (\Z_{\geq 0}[I])^m$, we set $\mathbf{s}=(\lambda_1,...,\lambda_m)$.
Let $\nu$ be the $n$ sequence of $i\in I$.
Khovanov-Lauda-Rouquier subalgebra $S_{C}(\lambda,\nu)$ is isomorphic to the redotted Webster algebra $W(\mathbf{s},n)$ defined in \cite{KLSY}.  
\end{theorem}

\begin{theorem}
We set the scalar parameters \eqref{scalar} as 
\begin{equation}
\nonumber
t(\tilde{i},\tilde{j})=\left\{\begin{array}{ll}
-1&\text{if}\quad(\tilde{i},\tilde{j})=(\bar{k},k) \in \bar{I}\times I \\[.5em]
1&\text{otherwise.}
\end{array}
\right.
\end{equation}
Taking the cyclotomic quotient and quotient of the ideal generated by $E_{\lambda_k}$ $(1\leq k\leq m)$, the quotient algebra $\overline{S}_{C}(\lambda,\nu)=S_{C}(\lambda,\nu)/I$ is isomorphic to the tensor product algebra defined by Webster \cite{Web}, where $I$ is the two-sided ideal generated by $E_{\lambda_k}$ $(1\leq k\leq m)$ and $\{e(\mathbf{i}) \in \mathrm{Seq}(\lambda,\nu)|\mathbf{i}_1\in I\}$.
\end{theorem}

\bibliographystyle{amsalpha}
\bibliography{bib_pDG_with_href}
%
%

\end{document}